\documentclass[]{amsart}
\usepackage{amsmath}
\usepackage[english]{babel}
\usepackage[utf8]{inputenc}
\usepackage{amsfonts}
\usepackage{amsthm}
\usepackage{color}
\numberwithin{equation}{section}

\newtheorem{thm}{Theorem}[section]

\newtheorem{prop}[thm]{Proposition}
\newtheorem{cor}[thm]{Corollary}
\newtheorem{lem}[thm]{Lemma}

\theoremstyle{remark}
\newtheorem{rmk}[thm]{Remark}
\theoremstyle{definition}
\newtheorem{defn}{Definition}[section]

\DeclareMathOperator{\E}{\mathbb{E}}

\DeclareMathOperator{\N}{\mathbb{N}}

\DeclareMathOperator{\cG}{\mathcal{G}}

\DeclareMathOperator{\cL}{\mathcal{L}}
\DeclareMathOperator{\cH}{\mathcal{H}}

\newcommand{\pd}[2]{\frac{\partial #1}{\partial #2}}
\newcommand{\pdsup}[3]{\frac{\partial^{#3} #1}{\partial #2^{#3}}}

\newcommand{\der}[2]{\frac{d #1}{d #2}}

\newcommand{\Norm}[2]{\left\Vert #1 \right\Vert_{#2}}

\author{Giacomo Ascione$^\ast$}
\address{$^\ast$ Dipartimento di Matematica e Applicazioni ``Renato Caccioppoli'', Università degli Studi di Napoli Federico II, 80126 Napoli, Italy}
\author{Yuliya Mishura$^\dagger$}
\address{$^\dagger$ Department of Probability Theory, Statistics and Actuarial Mathematics,Taras Shevchenko National University of Kyiv, Volodymyrska 64, Kyiv 01601, Ukraine}
\author{Enrica Pirozzi$^\ast$}
\email{giacomo.ascione@unina.it \\
	myus@univ.kiev.ua \\
	enrica.pirozzi@unina.it}
\title{Time-Changed fractional Ornstein-Uhlenbeck process\footnote{This paper is now published (in revised form) in Fract. Calc. Appl. Anal. Vol. 23, No 2 (2020), pp. 450-483, DOI: 10.1515/fca-2020-0022, and is available online at http://www.degruyter.com/view/j/fca , so please always cite it with the journal's coordinates}}
\begin{document}
	\maketitle
	\begin{abstract}
		We define a time-changed fractional Ornstein-Uhlenbeck process by composing a fractional Ornstein-Uhlenbeck process with the inverse of a subordinator. Properties of the moments of such process are investigated and the existence of the density is shown. We also provide a generalized Fokker-Planck equation for the density of the process.
	\end{abstract}

\section{Introduction}

The fractional Ornstein-Uhlenbeck (fOU) process is constructed as the solution of the stochastic differential equation (SDE) (\cite{cheridito2003fractional}), for $t\geq 0$ and $\theta>0$,
\begin{equation}\label{sde}
dU_t^H= -\frac{1}{\theta}U_t^Hdt+dB_t^H,
\end{equation}
where $B_t^H$ is the fractional Brownian motion (fBm) with Hurst parameter $H \in (0,1)$. This process is gaining an increasing attention due to its mathematical properties and stochastic features particularly suitable to model phenomena generated by processes with correlations. Indeed, the fOU process turns out to be useful to specialize models based on both OU-type processes and  fBm-type processes, because  it evolves according to the differential dynamics \eqref{sde}, typical for a classical OU process, and, at the same time, it preserves some stochastic  aspects of the fBm. 
More specifically, it can be viewed as a transformed fBm by the equation \eqref{sde}; in this sense, it is a process with a decay time $\theta$ towards the zero attractive level disrupted by a specialized noise, that is the fractional one $dB_t^H$ (\cite{Mis2008}). Nevertheless, the fOU preserves some properties of fBM: for instance, the long-range dependence is detectable in the asymptotic behavior of its covariance (\cite{mcap2019,cheridito2003fractional,kukush2017hypothesis}). Theoretical results about the standard OU process that are particularly useful for applications, have also been investigated for the fOU process providing a more general stochastic process and, in the same time, specializing and refining consolidated application models (see, e.g., \cite{dovidio2018}). 

Indeed, the behaviour of the covariance of the fBm turns out to be really useful to describe phenomena with memory. In the field of finance, for instance, models driven by the fBm are introduced to describe financial markets subject to memory effects (\cite{anh2005financial,gatheral2018volatility}). This application leads, for instance, to the study of fractional Cox-Ingersoll-Ross processes as square of fOU processes (\cite{mishura2017stochastic,mishura2018fractional}) and thus to further investigations of the first passage time of a fOU process through $0$. In physics, fBm models are used for instance to describe reaction-kinetics under subdiffusive dynamics (\cite{jeon2014first}), while in IT security, it is used to recognize Distributed Denial of Service attacks (\cite{li2008simulation}). In biology, in particular in the fields of the computational neuroscience, the ineffectiveness of the OU process to describe some neuronal dynamics with memory (\cite{shinomoto1999ornstein}) led to the definition of linear models of neuronal dynamics with different correlated noises (\cite{mbe2019,sakai1999temporally}), considering among them a fBm noise, thus leading to a fOU process (\cite{mcap2019}). All these applications are also supported by a growing theory on parameters' estimation and hypotheses testing on the fOU process (\cite{hu2010parameter,kukush2017hypothesis}) together with the study of the distribution of the maxima and the first passage time of the fBm (\cite{delorme2015maximum,jeon2014first,mishura2017stochastic}) and some advances in simulation of the fractional white noise (see e.g \cite{brouste2013parameter}).\par
On the other hand, this is not the only way to introduce strongly correlated processes. Indeed, in different contexts semi-Markov models are a growing up field. In finance, for instance, a semi-Markov extension of the Black-Scholes model can be adopted to describe option pricing (\cite{janssen2007semi}). In epidemiology, semi-Markov models are preferred to Markov ones to describe some infective dynamics, such as, for instance, AIDS (\cite{lefevre2016sir}). A common way to generate semi-Markov processes is considering a time-changed Markov process obtained from an additive Markov process (\cite{cinlar1974markov}). One of the simpler cases is given by a time-change made by using the inverse of a subordinator (see \cite{bertoin1999subordinators} for a more detailed description) which is independent from the starting Markov process. In the specific case of an $\alpha$-stable subordinator (\cite{meerschaert2013inverse}), this kind of construction is strictly linked to fractional calculus via a \textit{fractionalization} of the time-derivative in the backward Kolmogorov equation of the starting Markov process (\cite{meerschaert2011stochastic}). With these methods, for instance, fractional (in the sense of the time-change) Pearson diffusions are introduced and studied (\cite{leonenko2013fractional}) and in particular these processes have been shown to exhibit a long-range dependence (\cite{leonenko2013correlation}). An extension of this theory to general inverse subordinator has been made by introducing a generalized Caputo derivative as a particular pseudo-differential convolution operator (\cite{toaldo2015convolution}, for some other details see \cite{tenreiro2017}). By using this more general theory, one can construct general time-changed processes and study their generalized backward Kolmogorov equations, as done in \cite{gajda2015time}. The delay that is obtained via the time-change is quite useful to describe models in various field of research. In finance, fractional $M/M/1$ queues are seen to better adapt to some financial datasets (\cite{cahoy2015transient}). The same happens in population dynamics for fractional Yule processes (\cite{cahoy2014parameter}). In physics, such models are used to describe sub-diffusive behaviours of particles (\cite{weron2009anomalous}). Finally, in computational neuroscience, such a time-changed OU process could lead to spike trains that better fit some experimental data (\cite{ascione2017exit,ascione2019}).\par
Lately, there is a growing interest on time-changed non-Markov processes, in particular on the time-changed fBm. For instance, in \cite{gajda2014fokker} generalized Fokker--Planck equations for the time-changed fBm are studied while in \cite{mij2014} the correlation structure of such processes is exploited. A particular kind of time-changed fBm is also studied in \cite{kumar2019}. An important factor of such time-changed non-Markov processes is the interplay between the two different kind of memory, as one can observe in \cite{mij2014}. This interplay makes time-changed fBm-driven processes interesting tools for applications. Indeed, this kind of models have already shown their importance in finance, to generalize Black-Scholes models for option pricing (\cite{gu2012time,guo2014pricing}).
In this paper, we consider a fOU process time-changed by the inverse of a subordinator with Laplace exponent $\Psi$ independent from the starting fOU process: we denote with $U_H^{\Psi}(t)$ this time-changed process. Our aim is to study some properties of this process, such as the existence of its moments, their asymptotic behaviour, and the existence of the density. Such kind of process could be interesting for neuronal models (being both approaches useful to describe neuronal models, ad done in \cite{mcap2019} for the fOU and in \cite{ascione2019} for the time-changed OU), despite the difficulties related to the study of first passage times of fOU processes through fixed barriers. However, we focus on the case $H \in (1/2,1)$, to avoid some compliance in the representation of the variance. We aim to study the case $H \in (0,1/2)$ in future works. Moreover, we aim to introduce a generalized Fokker-Planck equation that is solved by its density. The problem is quite difficult because of the non-homogeneity in time of the diffusion coefficient of the original Fokker-Planck equation for the fOU. These difficulties are similar to the one encountered for the fBm in \cite{gajda2014fokker} and for a general Gaussian process in \cite{hah2011}. Here we need to focus on some additional properties of the variance of the fOU and then we need to introduce some ad-hoc operators in order to obtain the generalized Fokker-Planck equation (for generalized Fokker-Planck equation see also \cite{butko2018}).\par
The structure of the paper is the following:
\begin{itemize}
	\item In Section \ref{sec2} we introduce the basic notation and preliminaries on the fBm and the subordinators, then we define the time-changed fOU $U^\Psi_H$;
	\item In Section \ref{sec3} we show that the absolute moments of $U^\Psi_H$ are bounded and then we show  monotonicity and we exhibit the limit of such moments;
	\item In Section \ref{sec4} we use the characteristic function to show that the variables $U^\Psi_H(t)$ admit density for any $t>0$ under some assumption on the inverse subordinator;
	\item In Section \ref{sec5} we provide further properties of the variance function of a fOU without time-change, concerning in particular its Laplace transform and the behaviour of its first derivative;
	\item In Section \ref{sec6}, we introduce two operators that will be involved in the generalized Fokker-Planck equation proposed for the density of $U^\Psi_H$. In particular we show that the density of the fOU belongs to the domain of the first operator while the Lapalce transform of the density of the time-changed fOU belongs to the domain of the second one; in particular, by exploiting the relation between the Laplace transforms of the two densities, we are able to exploit a relation between the two operators; 
	\item Finally, Section \ref{gFP} is devoted to the introduction and the study of the generalized Fokker-Planck equation. In particular we prove that the density of the time-changed fOU is a \textit{mild solution} (in a sense that will be explained later) of such equation. Under additional hypothesis, we are also able to prove that such density is also a classical solution of the generalized Fokker-Planck equation. In Subsection~\ref{revisited} we re-consider the problem to find classical solutions of the generalized Fokker-Planck equation under less restrictive hypotheses, re-formulating it by using a different operator. Finally, we give some hypotheses on the Laplace exponent of the inverse subordinator under which the generalized Fokker-Planck equation can be rewritten as an integral equation.
\end{itemize}

\par
\par

\section{Definition of the Time-Changed Fractional Ornisten-Uhlenbeck process}\label{sec2}

Let $(\Omega, \mathcal{F}, P)$ be a complete probability space supporting all stochastic processes that will be considered below. Let us fix Hurst index $H \in \left(\frac{1}{2},1\right)$  and consider a fractional Brownian motion $B^H=\{B^H(t), t\ge 0\}$ with Hurst index $H$, that is, a Gaussian process with zero mean and covariance function $$\E[B^H(t)B^H(s)]=1/2(t^{2H}+s^{2H}-|t-s|^{2H}), s,t \in \mathbf{R}^+.$$
Let us also fix some number $\theta>0$ and introduce  the fOU process, starting from zero at zero (\cite{cheridito2003fractional}) as
\begin{equation*}
U_H(t)=e^{-\frac{t}{\theta}}\int_0^{t}e^{\frac{s}{\theta}}dB^H(s), t\ge 0.
\end{equation*}
Let us denote by   $\sigma=\{\sigma(y), y\ge 0\}$ a driftless subordinator with L\'{e}vy measure $\nu$ (\cite{bertoin1996Levy}). The L\'{e}vy measure $\nu$ is such that $\nu(-\infty,0)=0$ and fulfills the integrability condition
\begin{equation}\label{intest}
\int_0^{+\infty}(1 \wedge x)\nu(dx)<+\infty,
\end{equation}
and we have that $\E[e^{-\lambda \sigma(y)}]=e^{-y\Psi(\lambda)}$ with Laplace exponent $\Psi(\lambda)=\int_0^{+\infty}(1-e^{-\lambda x})\nu(dx)$, that is a Bernstein function (\cite{schilling2012bernstein}). Recall that Bernstein functions are invertible and belong to $C^1(0,+\infty)$ with completely monotone derivative. Moreover, they admit a unique extension to $\mathbf{H}:=\{\lambda \in \mathbf{C}: \ \Re(\lambda)\ge 0\}$ that is holomorphic in $\mathbf{H}^*:=\{\lambda \in \mathbf{C}: \ \Re(\lambda)>0\}$ (see \cite[Proposition $3.5$]{schilling2012bernstein}).  Let us also denote by $\cL$ the Laplace transform operator acting on the variable $t \in [0,+\infty)$ and by $\cL^{-1}$ its inverse. Let us suppose that $\nu(0,+\infty)=+\infty$. This is enough to ensure that  the process $\sigma(y)$ is strictly increasing (see \cite[Proposition $1.3$]{bertoin1999subordinators}).\par
Given a subordinator $\sigma(y)$, we can define the inverse subordinator $E=\{E(t), t\ge0\}$ as $E(t)=\inf\{y>0: \ \sigma(y)>t\}$.
Moreover, from $\nu(0,+\infty)=+\infty$ we know that $E(t)$ admits a probability density function $f_E(t,y)$ for any $t>0$.  \par
For the probability density function $f_E(t,y)$ it is well known the following Laplace transform formula (see \cite[Equation $3.13$]{meerschaert2008triangular}):
\begin{equation}\label{eq:Laptrans}
\cL[f_E(\cdot,y)](\lambda)=\frac{\Psi(\lambda)}{\lambda}e^{-y\Psi(\lambda)}.
\end{equation}
Finally,  let us consider a fOU process $U_H$ and an inverse subordinator $E$, independent from $U_H$. Then we define the time-changed fOU process as $U_H^\Psi(t):=U_H(E(t)), t\ge 0$.
\section{Absolute moments of the time-changed fOU process and their asymptotic behavior}\label{sec3}

Let us denote
\begin{equation*}
V_{n,H}(t)=\E[|U_H(t)|^n], n\in \mathbb{N}, \quad \mbox{and} \quad V^\Psi_{n,H}(t)=\E[|U_H^\Psi(t)|^n], n\in \mathbb{N}.
\end{equation*}
Recall, in particular, that
\begin{equation}\begin{gathered}\label{variance}
V_{2,H}(t)=
H(2H-1)\theta^{2H}\int_0^{\frac{t}{\theta}}\int_0^{\frac{t}{\theta}}e^{-s-u}  |u-s|^{2H-2}duds.
\end{gathered}\end{equation}
Since $U_H(t)$ is a Gaussian process, we can immediately present the higher moments of the even order:
\begin{align*}\begin{split}
V_{2n,H}(t)&=\frac{(2H \theta^{2H} (2H-1))^n\Gamma\left(\frac{2n+1}{2}\right)}{\sqrt{\pi}}\left(\int_0^{\frac{t}{\theta}}\int_0^{\frac{t}{\theta}}e^{-s-u}  |u-s|^{2H-2}duds\right)^n.
\end{split}\end{align*}
Returning to the variance, we see with evidence  that $t \mapsto V_{2,H}(t)$ is a continuous   strictly increasing in $t$ function with the limit value
\begin{equation*}
V_{2,H}(\infty)=\lim_{t \to +\infty}V_{2,H}(t)=\theta^{2H}H \Gamma(2H),
\end{equation*}
(see \cite{kukush2017hypothesis}). Hence, in particular, $V_{2,H}(t)$ is bounded by $V_{2,H}(\infty)$, and consequently    $V_{2n,H}(t)$ is   bounded by $V_{2n,H}(\infty):=\frac{2^n\Gamma\left(\frac{2n+1}{2}\right)}{\sqrt{\pi}}(V_{2,H}(\infty))^n$.

Now, let us establish some properties of the moments of time-changed fOU process.  In what follows, we shall use the   notation  $\cL[{V}_{2n,H}(\cdot)](\lambda)$, $\cL[{V}^{\Psi}_{2n,H}(\cdot)](\lambda)$ and $\cL[{f}_E(\cdot,y)](\lambda)$ for  the Laplace transform of $V_{2n,H}(\cdot)$, $V_{2n,H}^\Psi(\cdot)$ and $ f_{E}(\cdot,y)$, respectively.
\begin{lem} 1) If the density $f_E(t,y)$ of the inverse subordinator exists, then
	\begin{equation*}
	V_{2n,H}^\Psi(t)=\int_0^{+\infty}V_{2n,H}(y)f_E(t,y)dy\le V_{2n,H}(\infty),
	\end{equation*}
	which means that  the absolute moments of $U_H^\Psi$ are bounded, too. If the density $f_E(t,y)$ of the inverse subordinator is a continuous function in $t$, then $V_{2n,H}^\Psi(t)$ is continuous in $t$ as well.
	
	2) The moments $V_{2n,H}^\Psi(t)$ are increasing in $t$.
	
	3) For any $n \in \N$ we have that
	\begin{equation}\begin{gathered}\label{multiple}
	\lim_{t \to +\infty}V_{2n,H}^\Psi(t)=  H^n(2H-1)^n \Gamma(2nH+1) \theta^{2nH}\\\times \int_{[0,{+\infty})^{2n}} \frac{\prod_{i=1}^n|x_i-y_i|^{2H-2}}{\left(1+\sum_{i=1}^{n}x_i+\sum_{i=1}^{n}y_i\right)^{2nH+1}}
	\prod_{i=1}^{n}dx_idy_i=V_{2n,H}(\infty),
	\end{gathered}\end{equation}
	and the multiple integral, contained in $V_{2n,H}(\infty)$, is well defined.
\end{lem}
\begin{proof} Statement 1) is evident. In order to prove statement 2), consider $0\le s\le t$  and, for any Borel set $A \subseteq \mathbf{R}^2$, define the measure $\cH(s,t,A)=\mathbb{P}((E(s),E(t)) \in A)$. Since $0 \le E(s) \le E(t)$, then, defining $D=\{(x,y) \in \mathbf{R}^2: \ 0 \le x \le y\}$, we have that for any Borel set $A \subseteq \mathbf{R}^2$ it holds $\cH(s,t,A)=\cH(s,t,A \cap D)$ (in particular the measure $\cH(s,t,\cdot)$ is concentrated on $D$). Therefore $$V_{2n,H}^\Psi(t)-V_{2n,H}^\Psi(s)=\int_D(V_{2n,H}(y)-V_{2n,H}(x)) \cH(s,t,dxdy)\ge 0$$
	because $V_{2n,H}(y)-V_{2n,H}(x)\ge 0$ for any $(x,y) \in D$. \par
	Consider statement 3). In terms of Laplace transform,  we have from \eqref{eq:Laptrans} that
	\begin{equation}\label{eq:VPsi}
	\cL[{V}_{2n,H}^\Psi(\cdot)](\lambda)=\int_{0}^{+\infty}\frac{\Psi(\lambda)}{\lambda}V_{2n,H}(y)e^{-y\Psi(\lambda)}dy=
	\frac{\Psi(\lambda)}{\lambda} \cL[{V}_{2n,H}(\cdot)](\Psi(\lambda)).
	\end{equation}
	Now we need to determine $\cL[{V}_{2n,H}(\cdot)](\lambda)$. To do this, observe first that, with a change of variable,
	\begin{equation*}
	V_{2,H}(y)=H(2H-1)y^{2H}\int_0^1\int_0^1e^{-\frac{y}{\theta}(z+v)}|z-v|^{2H-2}dzdv.
	\end{equation*}
	Now, with the notation $A_n=H^n(2H-1)^n$ and $B_n=A_n\Gamma(2nH+1)$,  we have
	{\small\begin{equation*}\begin{gathered}
	\cL[{V}_{2n,H}(\cdot)](\lambda) =A_n \int_0^{+\infty}\left(y^{2H}\int_0^1\int_0^1  e^{-\frac{y}{\theta}(z+v)}|z-v|^{2H-2}dzdv\right)^ne^{-\lambda y}dy\\
	=A_n\int_{[0,1]^{2n}} \prod_{i=1}^{n}|z_i-v_i|^{2H-2}\int_0^{+\infty}e^{- \left (\frac{1}{\theta} \left(\sum_{i=1}^{n}z_i+\sum_{i=1}^{n}v_i\right)+\lambda\right)y}y^{2Hn}dy\prod_{i=1}^{n}dz_idv_i\\
	=B_n\lambda^{-1}\theta^{2nH}\int_{[0,\frac{1}{\theta \lambda}]^{2n}} \frac{\prod_{i=1}^n|x_i-y_i|^{2H-2}}{\left(1+\sum_{i=1}^{n}x_i+\sum_{i=1}^{n}y_i\right)^{2nH+1}}\prod_{i=1}^{n}dx_idy_i.
	\end{gathered}\end{equation*}}
	Now let us observe that for any $\beta>0$ and $n\in \mathbb{N}$
	\begin{multline*}
	\int_{[0,\beta]^{2n}} \frac{\prod_{i=1}^n|x_i-y_i|^{2H-2}}{\left(1+\sum_{i=1}^{n}x_i+\sum_{i=1}^{n}y_i\right)^{2nH+1}}\prod_{i=1}^{n}dx_idy_i\\
	\le \left(\int_0^\beta \int_0^\beta \frac{|x-y|^{2H-2}}{(1+x)^{H+\frac{1}{n}}(1+y)^{H+\frac{1}{n}}}dxdy\right)^n.
	\end{multline*}
	Concerning the integral $\int_0^\beta \int_0^\beta \frac{|x-y|^{2H-2}}{(1+x)^{H+\frac{1}{n}}(1+y)^{H+\frac{1}{n}}}dxdy$, we can directly apply to it the Hardy--Littlewood theorem, or   observe that, according to \cite[Theorem 1.9.1 and  Corollary 1.9.4]{Mis2008},  it holds that for a fractional Brownian motion $B^H$
	{\small\begin{equation*}\begin{gathered}
	\int_0^\beta \int_0^\beta \frac{|x-y|^{2H-2}}{(1+x)^{H+\frac{1}{ n}}(1+y)^{H+\frac{1}{ n}}}dx dy=\frac{1}{H(2H-1)}\E\left[\int_0^\beta \frac{1}{(1+x)^{ H+\frac{1}{n}}}dB^H(x)\right]^2\\ \le C(H)\left(\int_0^{\infty} \frac{ dx}{(1+x)^{1+\frac{1}{ Hn}}}\right)^{2H}\le C(H)(Hn)^{2H},
	\end{gathered}\end{equation*}}
	so, the multiple integral $V_{2n,H}^\Psi(\infty) $  in \eqref{multiple} is well defined.
	Furthermore,  we have   as $\lambda \to 0$:
	\begin{equation*}
	\cL[{V}_{2n,H}(\cdot)](\lambda)\simeq  V_{2n,H}(\infty) \lambda^{-1}.
	\end{equation*}
	and, from \eqref{eq:VPsi}, we also have
	\begin{equation*}
	\cL[{V}^\Psi_{2n,H}(\cdot)](\lambda)\simeq V_{2n,H}(\infty)\lambda^{-1}.
	\end{equation*}
	Thus, by Tauberian theorem for the Laplace transform, we have as $t \to +\infty$
	\begin{equation*}
	\int_0^{t} V_{2n,H}^\Psi(s)ds \simeq  V_{2n,H}(\infty)t.
	\end{equation*}
	Since, according to statement 2), $V_{2n,H}^\Psi(t)$ is increasing, the limit \linebreak $\lim_{t \to +\infty}V_{2n,H}^\Psi(t)$ is well defined.  Moreover, we can  use a modification  of the   l'Hospital's rule to the case when the integral $\int_0^{t} V_{2n,H}^\Psi(s)ds$ has a derivative $V_{2n,H}^\Psi(t)$ at all points except a countable set, and  get that  $$\lim_{t \to +\infty}\frac{\int_0^t V_{2n,H}^\Psi(s)ds}{t}=\lim_{t \to +\infty}V_{2n,H}^\Psi(t).$$ Therefore,
	\begin{equation*}
	\lim_{t \to +\infty}V_{2n,H}^\Psi(t)= V_{2n,H}(\infty).
	\end{equation*}
\end{proof}
\section{Existence of the density of $U_H^\Psi(t)$}\label{sec4}

Now we investigate the problem of the existence of probability density function of time-changed fractional Ornstein-Uhlenbeck process. Denote by   $p_H(t,x)$   probability density function of $U_H(t)$ and  by $p^\Psi_H(t,x)$ the probability density function of $U_H^\Psi(t)$, if this probability density function exists for all $t>0$.
\begin{prop}\label{prop1}
	Suppose that the density $f_E(t,y)$ of the inverse subordinator exists, and additionally, $\E[E^{-H}(t)]<+\infty$ for any $t>0$. Then the probability density function $p^\Psi_H(t,x)$ exists for all $t>0$ and satisfies the equation
	\begin{align*}
	p_H^\Psi(t,x)&=\int_0^{+\infty}p_H(y,x)f_E(t,y)dy.
	\end{align*}
\end{prop}

\begin{proof}
	Let us observe that the characteristic function $\varphi_H(t,z)=\E[e^{i z U_H(t)}]$ of fOU $U_H$ is given by $\varphi_H(t,z)=e^{-\frac{z^2}{2}V_{2,H}(t)}$. Define $\varphi_H^\Psi(t,z)=\E[e^{i z U_H^\Psi(t)}]$. From the independence of $E(t)$ and $U_H(t)$ we have
	\begin{equation}\label{charact}
	\varphi_H^\Psi(t,z)=\int_0^{+\infty}\varphi_H(y,z)f_E(t,y)dy=\int_0^{+\infty}e^{-\frac{z^2}{2}V_{2,H}(y)}f_E(t,y)dy.
	\end{equation}
	
	We want to show that $z \mapsto \varphi_H^\Psi(t,z)$ is an $L_1(\mathbf{R})$--function for any $t\ge 0$. In order to do this, observe that we can formally apply Fubini theorem, taking into account that all the functions are non-negative, and get that
	\begin{equation}\begin{gathered}\label{forFubini}
	\int_{\mathbf{R}}\varphi_H^\Psi(t,z)dz=\int_{\mathbf{R}}\int_0^{+\infty}e^{-\frac{z^2}{2}V_{2,H}(y)}f_E(t,y)dydz\\=(2\pi)^{1/2}\int_0^{+\infty} f_E(t,y)\frac{1}{\sqrt{V_{2,H}(y)}}dy.
	\end{gathered}\end{equation}
	According to equality \eqref{variance} for $V_{2,H}(t)$, we have ${V_{2,H}(y)\simeq  H(2H-1)y^{2H}}$ as $y \to 0$.  Therefore,  $\frac{V_{2,H}(y)}{y^{2H}}\ge C_1(H)>0$ for $y \in [0,1]$. It means that
	\begin{equation*}
	\int_0^{1} f_E(t,y)\frac{1}{\sqrt{V_{2,H}(y)}}dy\le C_1(H)^{-\frac{1}{2}}\E[ E^{-H}(t)]<+\infty.
	\end{equation*}
	Moreover, since $V_{2,H}(y)$ is an increasing function with $V_{2,H}(1)>0$,  the following upper   bound  holds:
	\begin{equation*}
	\int_{1}^{+\infty} f_E(t,y)\frac{1}{\sqrt{V_{2,H}(y)}}dy \le (V_{2,H}(1))^{-1/2}<+\infty.
	\end{equation*}
	If to summarize, we get that for any $t>0$ the Fourier transform $z \mapsto \varphi_H^\Psi(t,z)\in L_1(\mathbf{R})$. A standard application of L\'{e}vy inversion theorem implies the existence of $p^\Psi_H(t,x)$.
	Let us return to the equalities \eqref{charact}. Taking them into account, together with the integrability of the characteristic function  $\varphi_H^\Psi(t,z)$,  and applying
	inverse Fourier transform,  we get the following equation for the density $p_H^\Psi$:
	{\small \begin{equation*}
	p_H^\Psi(t,x)=\frac{1}{2\pi}\int_{\mathbf{R}}e^{-izx}\varphi_H^\Psi(t,z)dz
	=\frac{1}{2\pi}\int_{\mathbf{R}}e^{-izx}\int_0^{+\infty}\varphi_H(y,z)f_E(t,y)dydz.
	\end{equation*}}
	Now, the relations \eqref{forFubini} and   the subsequent upper bounds imply that conditions of the theorem guarantee the possibility to apply  the Fubini theorem to get that \begin{align*}
	\frac{1}{2\pi}\int_{\mathbf{R}}e^{-izx}\int_0^{+\infty}\varphi_H(y,z)f_E(t,y)dydz=\int_0^{+\infty}p_H(y,x)f_E(t,y)dy,
	\end{align*}
	and the proof follows. \end{proof}
\begin{rmk}
	If $E(t)$ is an inverse $\alpha$-stable subordinator, then for any $t>0$
	\begin{equation*}
	f_E(t,y)=\frac{t}{\alpha}y^{-1-\frac{1}{\alpha}}g_\alpha(ty^{-\frac{1}{\alpha}}),
	\end{equation*}
	where $g_\alpha$ is the density of a one-sided $\alpha$-stable random variable $S_\alpha$. Then
	\begin{equation*}
	\E[E(t)^{-H}]=\frac{t}{\alpha}\int_0^{+\infty}y^{-H}y^{-1-\frac{1}{\alpha}}g_\alpha(ty^{-\frac{1}{\alpha}})dy.
	\end{equation*}
	With the change of variable $z=ty^{-\frac{1}{\alpha}}$ we have $dz=-\frac{1}{\alpha} ty^{-1-\frac{1}{\alpha}}dy$, and $y=\left(\frac{z}{t}\right)^{-\alpha}$, therefore
	\begin{equation*}
	\E[E^{-H}(t)]= t^{-\alpha H } \int_0^{+\infty}z^{H\alpha}g_\alpha(z)dz= t^{-\alpha H} \E[S_\alpha^{H\alpha}]<+\infty,
	\end{equation*}
	since $H\alpha<\alpha$, and $S_\alpha$ has any moment of positive order less than $\alpha$.
\end{rmk}
From the integral representation of the characteristic function $\varphi_H^\Psi$, we also have the following corollary.
\begin{cor}
	Fix $n \in \mathbf{N}$. If $\E[E^{-(n+1)H}(t)]<+\infty$ for any $t>0$, then the density $p_H^\Psi(t,x)$ is differentiable $n$ times.
\end{cor}
\begin{proof}
	By using \cite[Theorem $9.2$]{rudin2006real}, it is only necessary to show that the function $z^n\varphi_H^\Psi(t,z)$ is in $L^1(0,+\infty)$. As before, we can formally apply Fubini's theorem, since inte integrand functions are non-negative, obtaining
	\begin{equation*}
	\int_0^{+\infty}z^n\varphi_H^\Psi(t,z)dz=C_n\int_0^{+\infty}f_E(t,y)(V_{2,H}(y))^{-\frac{n+1}{2}}dy
	\end{equation*}
	where
	\begin{equation*}
	C_n=\int_0^{+\infty}z^ne^{-\frac{z^2}{2}}dz.
	\end{equation*}
	Since $V_{2,H}(y)$ is an increasing function with $V_{2,H}(1)>0$, we have the following upper bound
	\begin{equation*}
	\int_1^{+\infty}f_E(t,y)(V_{2,H}(y))^{-\frac{n+1}{2}}dy\le (V_{2,H}(1))^{-\frac{n+1}{2}}<+\infty.
	\end{equation*}
	Moreover, since we know that $V_{2,H}(y)\ge C_1(H)y^{2H}$ for $y \in [0,1]$, we have that
	\begin{equation*}
	\int_0^{1}f_E(t,y)(V_{2,H}(y))^{-\frac{n+1}{2}}dy\le C_1(H)^{-\frac{n+1}{2}} \E[E^{-(n+1)H}(t)]<+\infty
	\end{equation*}
	concluding the proof.
\end{proof}
\begin{rmk}
	This is not the case of an inverse $\alpha$-stable subordinator. Indeed, by using the same manipulations as we did before, we have
	\begin{equation*}
	\E[E^{-(n+1)H}(t)]=t^{-\alpha(n+1)H}\E[S_\alpha^{(n+1)H\alpha}]
	\end{equation*}
	that, being $H>1/2$, is finite if and only if $n=0$.
\end{rmk}
\section{Some further properties of the variance function $V_{2,H}(\cdot)$}\label{sec5}

In this section we want to exploit some further properties of the variance function $V_{2,H}$ of the fractional Ornstein-Uhlenbeck process. First of all, let us recall, as done in Section \ref{sec3}, that the variance $V_{2,H}(t)$ is bounded hence its Laplace transform is well defined for any $\lambda \in \mathbf{H}^*$. Moreover,  we have the following lemma. Recall that $\cL[{V}_{2,H}(\cdot)](\lambda)$ stands for the Laplace transform of $V_{2,H}(\cdot)$.
\begin{lem}
	For any  $\lambda \in \mathbf{H}^*$  the following formula holds
	\begin{equation}\label{eq:5.1}
	\cL[{V}_{2,H}(\cdot)](\lambda)=\frac{2H\theta^{2H}\Gamma(2H)}{\lambda(\theta\lambda+2)(\theta\lambda +1)^{2H-1}}.
	\end{equation}
\end{lem}
\begin{proof}
	
	To obtain   formula \eqref{eq:5.1} for $\cL[{V}_{2,H}(\cdot)](\lambda)$, let us recall the following alternative representation of $V_{2,H}(t)$:
	\begin{equation}\label{eq:altvar}
	V_{2,H}(t)=H\left(\int_0^te^{-\frac{z}{\theta}}z^{2H-1}dz+e^{-\frac{2}{\theta}t}\int_0^{t}e^{\frac{z}{\theta}}z^{2H-1}dz\right),
	\end{equation}
	as given in \cite{kukush2017hypothesis}. Thus we have, by using Fubini's theorem,
	\begin{align*}
	&\cL[{V}_{2,H}(\cdot)](\lambda)=\int_0^{+\infty}e^{-\lambda t}V_{2,H}(t)dt\\
	&\qquad =H\int_0^{+\infty}z^{2H-1}\left(e^{-\frac{z}{\theta}}\int_z^{+\infty}e^{-\lambda t}dt+e^{\frac{z}{\theta}}\int_z^{+\infty}e^{-\left(\lambda+\frac{2}{\theta}\right)t}dt\right)dz\\
	&\qquad=\frac{2H\theta^{2H}(\theta\lambda+1)}{\lambda(\theta\lambda+2)(\theta\lambda+1)^{2H}}\int_0^{+\infty}e^{-y}y^{2H-1}dy =\frac{2H\theta^{2H}\Gamma(2H)}{\lambda(\theta\lambda+2)(\theta\lambda+1)^{2H-1}},
	\end{align*}
	where we used the change of variable $y=\left(\lambda+\frac{1}{\theta}\right)z$.
\end{proof}
Moreover, let us give some information on the derivative of the variance $V_{2,H}(t)$.
\begin{lem}\label{lemma:asymVpr}
	The function $V_{2,H}\in C^1[0,+\infty)$ and its derivative  satisfies the following relations:
	\begin{itemize}
		\item[$(i)$] $\lim_{t \to +\infty} {e^{\frac{t}{\theta}}t^{2-2H}}{V'_{2,H}(t)}=2H(2H-1)\theta$; in particular, it means that $\lim_{t \to +\infty}V'_{2,H}(t)=0$.
		\item[$(ii)$] $\lim_{t \to 0}\frac{V'_{2,H}(t)}{t^{2H-1}}=2H$; in particular, it means that $\lim_{t \to 0}V'_{2,H}(t)=0$.
	\end{itemize}
\end{lem}
\begin{proof}
	The fact that $V_{2,H}(t)$ is a $C^1$ function in $(0,+\infty)$ follows easily from equation \eqref{eq:altvar}. Indeed, differentiating this  equation and then integrating by parts, we get that or any $t>0$
	\begin{align}\label{deriv}
	\begin{split}
	V'_{2,H}(t)=2H(2H-1)e^{-\frac{2}{\theta}t}\int_0^te^{\frac{z}{\theta}}z^{2H-2}dz.
	\end{split}
	\end{align}
	The derivative at zero can be calculated using the L'Hospital's rule:
	$$V'_{2,H}(0)=\lim_{t \to 0}\frac{V_{2,H}(t)}{t}=2H\lim_{t \to 0}t^{2H-1}=0,$$
	and it follows from \eqref{deriv} that $\lim_{t\downarrow 0}V'_{2,H}(t)=0.$
	So, $V_{2,H}\in C^1[0,+\infty)$. Now, to obtain $(i)$, we  use again the L'Hospital's rule:
	\begin{align*}
	\lim_{t \to +\infty}\frac{V'_{2,H}(t)}{e^{-\frac{t}{\theta}}t^{2 H-2}}=\lim_{t \to +\infty}\frac{2H(2H-1)\theta}{1+2\theta(H-1)t^{-1}}=2H(2H-1)\theta.
	\end{align*}
	
	Further, to obtain $(ii)$, let us use the L'Hospital's rule once again and evaluate:
	\begin{align*}
	\lim_{t \to 0}\frac{V'_{2,H}(t)}{t^{2H-1}}=\lim_{t \to 0}\frac{2H(2H-1)  }{ e^{\frac{t}{\theta}}\left(\frac{2}{\theta}t+(2H-1)\right)}=2H.
	\end{align*}
\end{proof}
With such asymptotics, we can easily establish the following fact.
\begin{cor}
	$V'_{2,H}(t)$ is in $L^2(0,+\infty)$.
\end{cor}
\begin{proof}
	Indeed, $V \in C^1([0,+\infty))$,  $\lim_{t \to 0}V'_{2,H}(t)=0$ and, as \linebreak ${t \to +\infty}$, we have that $V'_{2,H}(t)\simeq t^{2H-2}e^{-\frac{t}{\theta}}$.
\end{proof}
Moreover, we also have the following Laplace transform formula for $V'_{2,H}(t)$.
\begin{lem}
	$V'_{2,H}(\cdot)$ is Laplace transformable for any  $\lambda \in \mathbf{C}$ such that $\Re(\lambda)>-\frac{1}{\theta}$. In particular,
	\begin{equation}\label{eq:laptrVpr}
	\cL[V'_{2,H}(\cdot)](\lambda)=\frac{2H\theta^{2H}\Gamma(2H)}{(\theta\lambda+2)(\theta\lambda +1)^{2H-1}}.
	\end{equation}
	Moreover, $\cL[V'_{2,H}(\cdot)](\lambda)$ is  holomorphic  in  $\left\{\lambda \in \mathbf{C}: \Re(\lambda)>-\frac{1}{\theta}\right\}$  and for any $c \in \mathbf{R}$ such that  $c>-\frac{1}{\theta}$  the function $\omega \in \mathbf{R} \mapsto \cL[V'_{2,H}(\cdot)](c+i\omega)$ is in $L^1(\mathbf{R})\cap L^2(\mathbf{R})$.
\end{lem}
\begin{proof}
	Consider $\lambda \in \mathbf{C}$ such that $\Re(\lambda)>-\frac{1}{\theta}$. We have
	\begin{equation*}
	\int_0^{+\infty} e^{-\lambda t}V_{2,H}'(t)dt=\int_0^1 e^{-\lambda t}V_{2,H}'(t)dt+\int_1^{+\infty}e^{-\lambda t}V_{2,H}'(t)dt
	\end{equation*}
	For $t \in [0,1]$, let us just observe that, since $V_{2,H}'(t)$ is continuous, there exists a constant $C_1$ such that
	\begin{equation*}
	|e^{-\lambda t}V_{2,H}'(t)|\le C_1e^{\frac{t}{\theta}}
	\end{equation*}
	that is in $L^1([0,1])$. For $t \in [1,+\infty)$, let us recall from Lemma \ref{lemma:asymVpr}, property $(i)$, that there exists a constant $C_2$ such that
	\begin{equation*}
	V_{2,H}'(t)\le C_2 t^{2H-2}e^{-\frac{t}{\theta}}
	\end{equation*}
	and then
	\begin{equation*}
	|e^{-\lambda t}V_{2,H}'(t)|\le C_2 t^{2H-2}e^{-\left(\Re(\lambda)+\frac{1}{\theta}\right)t}
	\end{equation*}
	where $t^{2H-2}e^{-\left(\Re(\lambda)+\frac{1}{\theta}\right)t}$ is in $L^1(1,+\infty)$ since $\Re(\lambda)+\frac{1}{\theta}>0$. Hence \linebreak $\cL[V'_{2,H}(\cdot)](\lambda)$ is well defined for any $\lambda \in \mathbf{C}$ such that $\Re(\lambda)>-\frac{1}{\theta}$. Moreover, since $V'_{2,H}(t)$ is continuous, it belongs to $L^1_{loc}(0,+\infty)$ and then, from \cite[Theorem $1.5.1$]{wolfgang2002vector}, we know that $\cL[V'_{2,H}(\cdot)](\lambda)$ is holomorphic in \linebreak $\left\{\lambda \in \mathbf{C}: \ \Re(\lambda)>-\frac{1}{\theta}\right\}$. Equation \eqref{eq:laptrVpr} for $\lambda \in \mathbf{H}^*$ follows from the relation
	\begin{equation*}
	\cL[V'_{2,H}(\cdot)](\lambda)=\lambda \cL[{V}_{2,H}(\cdot)](\lambda)
	\end{equation*}
	and the fact that it holds also for   $\Re(\lambda)>-\frac{1}{\theta}$ follows from the fact that $\cL[V'_{2,H}(\cdot)](\lambda)$ is analytic. Finally, consider $\omega \in \mathbf{R} \mapsto \cL[V'_{2,H}(\cdot)](c+i\omega)$. For $\omega \in [-1,1]$, $\cL[V'_{2,H}(\cdot)](c+i\omega)$ is bounded, while for $\omega \in (-\infty,-1)\cup(1,+\infty)$ we have
	\begin{equation*}
	|\cL[V'_{2,H}(\cdot)](c+i\omega)|=\frac{2H\theta^{2H}\Gamma(2H)}{|\omega|^2\left|\frac{\theta c+2}{\omega}+\theta i\right|\left|\frac{\theta c+1}{\omega}+\theta i\right|^{2H-1}}\le \frac{2H\theta^{2H}\Gamma(2H)}{|\omega|^2}.
	\end{equation*}
	Hence $\omega \in \mathbf{R} \mapsto \cL[V'_{2,H}(\cdot)](c+i\omega)$ is in $L^1(\mathbf{R})\cap L^2(\mathbf{R})$ for any $c\ge 0$. 
\end{proof}
\section{Some operators involving $V_{2,H}$}\label{sec6}

In this section we will introduce some operators involving the variance function $V_{2,H}$. These operators will be used in the next section to introduce a generalized Fokker-Planck equation.\par
Let us define, for any measurable function $u:[0,+\infty)\times R$, the operator
\begin{equation*}
L(u)(\lambda,x)=\int_0^{+\infty}e^{-\lambda s}V'_{2,H}(s)u(s,x)ds, \ \lambda \in \mathbf{H}^*, x\in I\subset R.
\end{equation*}
Denote by $\mathcal{D}(L,I)$ the domain of such operator, i.e., the set of measurable functions $u:[0,+\infty)\times \mathbf{R}$, for which $L(u)(\lambda,x)$ is well defined for any $x\in I$. Now we investigate the belonging of some particular functions to  domains $\mathcal{D}(L,I)$, for respective $I$. Introduce the notation  $\mathbf{R}^*=\mathbf{R}\setminus\{0\}$.
\begin{lem}\label{lemma:deriv}  The following relations hold:
	\begin{itemize}
		\item[$(iii)$] $p_H(t,x)\in \mathcal{D}(L,\mathbf{R})$;
		\item[$(iv)$] $\pd{p_H}{x}(t,x) \in \mathcal{D}(L,\mathbf{R})$;
		\item[$(v)$] For any $x\in \mathbf{R}^*$ and $\lambda \in \mathbf{H}$ we have that
		\begin{equation}\label{first-deriv}
		L\left(\pd{p_H}{x}\right)(\lambda,x)=\pd{}{x}L(p_H)(\lambda,x),
		\end{equation}
		this formula holds also for $x=0$ when $\lambda \in \mathbf{H}^*$;
		\item[$(vi)$] $\pdsup{p_H}{x}{2}(t,x) \in \mathcal{D}(L,\mathbf{R}^*)$ and
		\begin{equation*}
		L\left(\pdsup{p_H}{x}{2}\right)(\lambda,x)=\pdsup{}{x}{2}L(p_H)(\lambda,x).
		\end{equation*}
	\end{itemize}
\end{lem}
\begin{proof}
	Without loss of generality, let us prove our lemma for $\lambda \in \mathbf{R}$ with  $\lambda\ge 0$. Recall that the density $p_H(t,x)$ of the Gaussian r.v. $U_H(t)$ equals
	\begin{equation}\label{density}
	p_H(t,x)=\frac{1}{\sqrt{2\pi V_{2,H}(t)}}e^{-\frac{x^2}{2V_{2,H}(t)}}.
	\end{equation}
	Let us fix  $x \in \mathbf{R}$  and   observe that
	\begin{equation*}
	L(p_H)(\lambda,x)=\int_0^{1}e^{-\lambda t}V'_{2,H}(t)p_H(t,x)dt+\int_1^{+\infty}e^{-\lambda t}V'_{2,H}(t)p_H(t,x)dt.
	\end{equation*}
	Observe that since $V'_{2,H}(t)\simeq t^{2H-1}$ and $V_{2,H}(t)\simeq t^{2H}$ as $t \to 0$, we have that there exists a constant $C_1(H)$ such that for any $t \in [0,1]$
	\begin{equation*}
	\frac{V'_{2,H}(t)}{\sqrt{2\pi V_{2,H}(t)}} \le C_1(H) t^{H-1},
	\end{equation*}
	and, consequently,
	\begin{equation*}
	\int_0^{1}e^{-\lambda t}V'_{2,H}(t)p_H(t,x)dt\le C_1(H)\int_0^1 t^{H-1}<+\infty.
	\end{equation*}
	Furthermore,  since $V'_{2,H}(t)\simeq e^{-\frac{t}{\theta}}t^{2H-2}$ as $t \to +\infty$ and for $t \in [1,+\infty)$ we have $V_{2,H}(t)\ge V_{2,H}(1)$, there exists a constant $C_2(H)$ such that
	\begin{equation*}
	\frac{V'_{2,H}(t)}{\sqrt{2\pi V_{2,H}(t)}}\le C_2(H)t^{2H-2}e^{-\frac{t}{\theta}},
	\end{equation*}
	and then
	\begin{equation*}
	\int_1^{+\infty}e^{-\lambda t}V'_{2,H}(t)p_H(t,x)dt \le C_2(H)\int_1^{+\infty}t^{2H-2}e^{-\frac{t}{\theta}}dt<+\infty.
	\end{equation*}
	It means  that $L(p_H)(\lambda,x)<+\infty$ for any $x \in \mathbf{R}$ and  $\lambda \ge 0$, whence  $(iii)$ follows.
	For the derivative, let us observe that
	\begin{equation*}
	\pd{p_H(t,x)}{x}=-\frac{x}{ \sqrt{2\pi} V^{3/2}_{2,H}(t) }e^{-\frac{x^2}{2V_{2,H}(t)}}.
	\end{equation*}
	In particular, for $x=0$ we have $\pd{p_H}{x}(t,0)=0$ for any $t>0$ and then
	\begin{equation*}
	\int_0^{+\infty}e^{-\lambda t}V'_{2,H}(t)\pd{p_H}{x}(t,0)dt=0.
	\end{equation*}
	Now let us consider $x  \not = 0$. Observe that
	{\small \begin{equation*}
	L\left(\pd{p_H}{x}\right)(\lambda,x )=\int_0^{1}e^{-\lambda t}V'_{2,H}(t)\pd{p_H}{x}(t,x )dt+\int_1^{+\infty}e^{-\lambda t}V'_{2,H}(t)\pd{p_H}{x}(t,x )dt.
	\end{equation*}}
	Using as before Lemma \ref{lemma:asymVpr} $(ii)$ and the fact that as $t \to 0$ $V_{2,H}(t)\simeq t^{2H}$ we have that there exist two constants $C_3(H)$ and $C_4(H)$ such that, for any $t \in [0,1]$,
	\begin{equation*}
	\left|e^{-\lambda t}V'_{2,H}(t)\pd{p_H}{x}(t,x)\right|\le \frac{C_3(H)x}{t^{1+H}}e^{-\frac{x^2}{C_4(H)t^{2H}}},
	\end{equation*}
	where the function at the right-hand side is a $L^1(0,1)$ function.
	Moreover, the  obtained  upper bounds  imply that $\int_0^{1}e^{-\lambda t}V'_{2,H}(t)\pd{p_H}{x}(t,x )dt$ converges uniformly in the interval $(\frac{ {x  }}{2}, \frac{3x}{2})$, for any  $x>0$, and in the interval $(\frac{ {3x  }}{2}, \frac{x}{2})$ for any $x<0$. For $t \in [1,+\infty)$, let us observe, as before,  that $V'_{2,H}(t) \simeq t^{2H-2}e^{-\frac{t}{\theta}}$ as $t \to +\infty$ and $V_{2,H}(t)\ge V_{2,H}(1)$, therefore, for some $C_5(H)>0$,
	\begin{equation*}
	\left|e^{-\lambda t}V'_{2,H}(t)\pd{p_H}{x}(t,x)\right|\le C_5(H)xt^{2H-2}e^{-\frac{t}{\theta}}
	\end{equation*}
	which is integrable.  Hence we have $L\left(\pd{p_H}{x}\right)(\lambda,x)<+\infty$ for any $x \in \mathbf{R}$ and $\lambda \ge 0$, whence $(iv)$ follows. Moreover, as before,  $\int_1^{\infty}e^{-\lambda t}V'_{2,H}(t)\pd{p_H}{x}(t,x )dt$ converges uniformly in the  interval $(\frac{ {x  }}{2}, \frac{3x}{2})$, for any $x>0$, and in the interval $(\frac{ {3x  }}{2}, \frac{x}{2})$ for any $x<0$.
	Now turn to $(v)$, but for $x\neq 0$ this statement is an immediate consequence of convergence of integral $L(p_H)(\lambda,x)$ and uniform convergence of the integral $L\left(\pd{p_H}{x}\right)(\lambda,x)$ in some interval surrounding $x$, as it was just stated, and  the theorem on the differentiation of improper integral in the parameter.
	
	Let $x=0$. Then
	\[
	\pd{}{x}\bigl(L(p_H)(\lambda,x)\bigr)\biggr|_{x=0}
	=\lim_{x\to0}\int_0^\infty \frac{e^{-\lambda t} V'_{2,H}(t)}{\sqrt{2\pi V_{2,H}(t)}} \left[\frac{e^{-\frac{x^2}{2V_{2,H}(t)}}-1}{x}\right]dt,
	\]
	and for any $t>0$,
	\[
	\pi(x):=\frac{e^{-\frac{x^2}{2V_{2,H}(t)}}-1}{x} \to 0
	\quad\text{as }x\to0
	\]
	and $\pi(x)$ is bounded in absolute value by
	\[
	\frac{|x|}{2V_{2,H}(t)} \le \frac{1}{2V_{2,H}(t)}
	\quad\text{for } |x|<1,
	\]
	supplying that  $e^{-\lambda t} \frac{V'_{2,H}(t)}{\sqrt{2\pi V^3_{2,H}(t)}}$ is an integrable dominant  when $\lambda>0$.
	It means by the Lebesgue dominant convergence theorem that
	\[
	\pd{}{x}\bigl(L(p_H)(\lambda,x)\bigr)\biggr|_{x=0}
	= 0 = L\left(\pd{p_H}{x}\right)(\lambda,0).
	\]
	
	Now let us consider the second derivative. We have
	\begin{equation*}
	\pdsup{p_H}{x}{2}(t,x)=\left(-\frac{1}{ \sqrt{2\pi} V^{3/2}_{2,H}(t)} +\frac{x^2}{ \sqrt{2\pi} V^{5/2}_{2,H}(t)}\right) e^{-\frac{x^2}{2V_{2,H}(t)}}.
	\end{equation*}
	Fix $x \in \mathbf{R}^*$ and, as usual, divide the corresponding integral into two parts:
	{\small \begin{equation}\label{second-der}
	L\left(\pdsup{p_H}{x}{2}\right)(\lambda,x)= \int_0^{1}e^{-\lambda t}V'_{2,H}(t)\pdsup{p_H}{x}{2}(t,x)dt+\int_1^{+\infty}e^{-\lambda t}V'_{2,H}(t)\pdsup{p_H}{x}{2}(t,x)dt.
	\end{equation}}
	Working as before, we have, for some $C_6(H)>0$ and for any $t \in [0,1]$,
	\begin{equation}\label{upp-bound-1}
	\left|e^{-\lambda t}V'_{2,H}(t)\pdsup{p_H}{x}{2}(t,x)\right|\le \left(\frac{C_3(H)}{t^{1+H}}+\frac{C_6(H)x^2}{t^{3H+1}}\right)e^{-\frac{x^2}{C_4(H)t^{2H}}}
	\end{equation}
	which is a $L^1(0,1)$ function.\par
	For $t \in [1,+\infty)$, arguing as before we have, for some $C_7(H)>0$,
	\begin{equation}\label{upp-bound-2}
	\left|e^{-\lambda t}V'_{2,H}(t)\pdsup{p_H}{x}{2}(t,x)\right|\le C_7(H)t^{2H-2}e^{-\frac{t}{\theta}}
	\end{equation}
	which is a $L^1(1,+\infty)$ function.  Hence we have that $L\left(\pdsup{p_H}{x}{2}\right)(\lambda,x)<+\infty$ for any $x \in \mathbf{R}^*$ and  $\lambda\ge0$.
	In order to prove the second part of property $(vi)$,  we need to differentiate \eqref{first-deriv} in $x \in \mathbf{R}^*$. It can be performed similarly as we obtained \eqref{first-deriv} for  $x \in \mathbf{R}^*$, applying the theorem on the differentiation of improper integral in the parameter, only to note that both integrals in the right -hand side of \eqref{second-der} converge uniformly in some interval surrounding $x \in \mathbf{R}^*$, due to the upper bounds \eqref{upp-bound-1} and \eqref{upp-bound-2}.
\end{proof}
\begin{rmk}
	Let us observe that for $x=0$ we have
	\begin{equation*}
	\pdsup{p_H}{x}{2}(t,0)=-\frac{1}{ \sqrt{2\pi} V^{3/2}_{2,H}(t)},
	\end{equation*}
	so, for  $\lambda \ge 0$  we have that
	\begin{equation}\label{asympt-2}
	\left|e^{-\lambda t}V'_{2,H}(t)\pdsup{p_H}{x}{2}(t,0)\right|=\frac{e^{-\lambda t}V'_{2,H}(t)}{ \sqrt{2\pi} V^{3/2}_{2,H}(t)},
	\end{equation}
	and for  $t \to 0$ the   function in the right -hand side of the \eqref{asympt-2} behaves like $t^{-1-H}$, the latter being a non-integrable   in $(0,+\infty)$ function.
\end{rmk}
Now let us show that $p_H$ and $p_H^\Psi$ are bounded functions for fixed $x\in \mathbf{R}^*$.
\begin{lem}\label{lemma:contr}
	For any $x\in \mathbf{R}^*$ there exists a constant $C_H(x)$ such that
	\begin{equation*}
	\sup_{t \in (0,+\infty)}\{p_H(t,x),p_H^\Psi(t,x)\}\le C_H(x).
	\end{equation*}
\end{lem}
\begin{proof}
	Let us recall that, since as $t \to 0$ $V_{2,H}(t)\simeq t^{2H}$, there exists two constants $C_1(H)$ and $C_2(H)$ such that for any $t \in [0,1]$ we have
	\begin{equation*}
	\sqrt{2\pi V_{2,H}(t)}\ge C_1(H)t^{H} \qquad\text{and}\qquad 2V_{2,H}(t)\le C_2(H)t^{2H}.
	\end{equation*}
	From these bounds, combined with \eqref{density}, we get that for $t \in (0,1]$ and $x \not = 0$
	\begin{equation*}
	0 \le p_H(t,x)\le \frac{1}{C_1(H)t^H}e^{-\frac{x^2}{C_2(H)t^{2H}}}
	\end{equation*}
	and then, taking the limit as $t \to 0$ we have $\lim_{t \to 0}p_H(t,x)=0$.
	We also have
	\begin{equation*}
	\lim_{t \to +\infty} p_H(t,x)=\frac{1}{\sqrt{2\pi V_{2,H}(\infty)}}e^{-\frac{x^2}{2V_{2,H}(\infty)}}<+\infty.
	\end{equation*}
	Since $p_H(t,x)$ is continuous in $(0,+\infty)$, we have that for any $x\in \mathbf{R}^*$ there exists a constant $C_H(x)$ such that $p_H(t,x)\le C_H(x)$ for any $t \in (0,+\infty)$.\par
	Moreover, we have
	\begin{equation*}
	p_H^\Psi(t,x)=\int_0^{+\infty}p_H(y,x)f_E(y,t)dy\le C_H(x)\int_0^{+\infty}f_E(y,t)dy=C_H(x),
	\end{equation*}
	and this completes the proof.
\end{proof}
From now on, in order to simplify, we introduce the notation for Laplace transform of  $p_H(t,x)$ in $t$: $\overline{p}_H(\lambda,x):=\cL[p_H(\cdot,x)](\lambda)$. Applying   Lemma \ref{lemma:contr}, we   obtain the following useful corollary.
\begin{cor}
	The Laplace transform $\overline{p}_H(\lambda,x)$ is well defined for any $\lambda \in \mathbf{H}^*$ and $x \in \mathbf{R}^*$. Moreover, for fixed $c>0$ and $x \in \mathbf{R}^*$ the function
	\begin{equation*}
	\omega \in \mathbf{R} \mapsto \overline{p}_H(c+i\omega,x)
	\end{equation*}
	is bounded.
\end{cor}
\begin{proof}
	The first assertion follows easily from Lemma \ref{lemma:contr}, since we have for $\lambda \in \mathbf{R}$ with $\lambda>0$
	\begin{equation*}
	\int_0^{+\infty}e^{-\lambda t}p_H(t,x)dt\le \frac{C_H(x)}{\lambda}<+\infty.
	\end{equation*}
	This upper bound also gives the second assertion, since we have for $\lambda=c+i\omega$:
	\begin{equation*}
	|\overline{p}_H(c+i\omega,x)|\le \int_0^{+\infty}e^{-ct}p_H(t,x)dt\le \frac{C_H(x)}{c}.
	\end{equation*}
\end{proof}
We need another control on the growth of $p_H(t,x)$.
\begin{lem}
	Fix $x\in \mathbf{R}^*$. Then the map $t \mapsto p_H(t,x)$ is Lipschitz. Moreover, for $x \in \mathbf{R}^*$, $\pd{p_H}{t}(t,x)$ is Laplace transformable in $\mathbf{H}^*$ and then for any $\lambda \in \mathbf{H}^*$
	\begin{equation}\label{eq:Laptrpprime}
	\cL\left[\pd{p_H}{t}(\cdot,x)\right](\lambda)=\lambda \overline{p}_H(\lambda,x).
	\end{equation}
\end{lem}
\begin{proof}
	Observe that
	\begin{equation*}
	\pd{p_H}{t}(t,x)=\frac{V'_{2,H}(t)}{2}\left(\frac{-V_{2,H}(t)+x^2}{V^2_{2,H}(t)\sqrt{2\pi V_{2,H}(t)}}\right)e^{-\frac{x^2}{2V_{2,H}(t)}}
	\end{equation*}
	that is continuous in $(0,+\infty)$. In particular we have $\lim_{t \to +\infty}\pd{p_H}{t}(t,x)=0$. Now fix $x \in \mathbf{R}^*$. By Lemma \ref{lemma:asymVpr}, property $(i)$, we know that for $t \in (0,1)$ there exists a constant $C_1(H)$ such that $V'_{2,H}(t)\le 2C_1(H)t^{2H-1}$. Moreover, since as $t \to 0^+$ we have $V_{2,H}(t)\simeq t^{2H}$, there exist two constants $C_2(H)$ and $C_3(H)$ such that for $t \in (0,1)$
	\begin{equation*}
	\left|\frac{V_{2,H}(t)+x^2}{V^2_{2,H}(t)\sqrt{2\pi V_{2,H}(t)}}\right|\le C_2(H)\frac{t^{2H}+x^2}{t^{5H}}e^{-\frac{x^2}{C_3(H)t^{2H}}}
	\end{equation*}
	and then we have, for $t \in (0,1)$
	\begin{equation*}
	0\le \left|\pd{p_H}{t}(t,x)\right|\le C_4(H)\frac{t^{2H}+x^2}{t^{5H}}e^{-\frac{x^2}{C_3(H)t^{3H+1}}}
	\end{equation*}
	where $C_4(H)=C_1(H)C_2(H)$. Finally, taking the limit as $t \to 0^+$ we have $\lim_{t \to 0^+}\pd{p_H}{t}(t,x)=0$ for any $x \in \mathbf{R}^*$. Thus we have that $t \mapsto \pd{p_H}{t}(t,x)$ is bounded for any $x \in \mathbf{R}^*$ and $t \mapsto p_H(t,x)$ is Lipschitz.\par
	The fact that $t \mapsto \pd{p_H}{t}(t,x)$ is Laplace transformable in $\mathbf{H}^*$ for any $x \in \mathbf{R}^*$ follows from the fact that it is bounded. Finally, formula \eqref{eq:Laptrpprime} follows from the Laplace transform of the derivative and the fact that $p_H(0,x)=0$ for any $x \in \mathbf{R}^*$.
\end{proof}

Now we are ready to prove the following proposition.
\begin{prop}\label{proposition:59}
	Fix  $c_1<0<c_2$ such that $c_1-c_2>-\frac{1}{\theta}$. Then, for $x \in \mathbf{R}^*$ and $\lambda \in \mathbf{H}^*$, we have the following equality
	\begin{align}\label{eq:LapcarL}
	\begin{split}
	&L(p_H)(\lambda,x)=\frac{1}{4\pi^2}\int_{0}^{+\infty}e^{-\lambda t}\lim_{R_2 \to +\infty}\int_{-\infty}^{+\infty}e^{(c_1+iw)t}\\&\times\int_{-R_2}^{R_2}\cL[V'_{2,H}(\cdot)](c_1-c_2+i(w-v))\overline{p}_H(c_2+iv,x)dvdwdt.
	\end{split}
	\end{align}
\end{prop}
\begin{proof}
	Let us fix $R_1,R_2>0$. Consider the value
	{\small\begin{equation*}
	I(R_1,R_2):=\int_{-R_1}^{R_1}e^{(c_1+iw)t} \int_{-R_2}^{R_2}\cL[V'_{2,H}(\cdot)]
	(c_1-c_2+i(w-v))\overline{p}_H(c_2+iv,x)dvdw.
	\end{equation*}}
	Since all the involved functions are bounded for fixed $c_1,c_2$ and $x$, we can use Fubini theorem to obtain
	{\small\begin{align}\label{eq:intpass0}
	\begin{split}
	&I(R_1,R_2)= \int_{-R_2}^{R_2}\overline{p}_H(c_2+iv,x) \int_{-R_1}^{R_1}e^{(c_1+iw)t}\cL[V'_{2,H}(\cdot)](c_1-c_2+i(w-v))dwdv\\
	&= \int_{-R_2}^{R_2}e^{(c_2+iv)t}\overline{p}_H(c_2+iv,x) \int_{-R_1-v}^{R_1-v}e^{((c_1-c_2)+iu)t}\cL[V'_{2,H}(\cdot)](c_1-c_2+iu)dudv.
	\end{split}
	\end{align}}
	Recall that, since $V'_{2,H}$ is in $L^2(0,+\infty)$, by Paley-Wiener theorem (see \cite[Theorem $19.2$]{rudin2006real})
	\begin{equation}\label{PWthm}
	\lim_{R_1 \to +\infty}\frac{1}{2\pi}\int_{-R_1-v}^{R_1-v}e^{((c_1-c_2)+iu)t}\cL[V'_{2,H}(\cdot)](c_1-c_2+iu)du=V'_{2,H}(t).
	\end{equation}
	Furthermore, we have shown that $u \mapsto \cL[V'_{2,H}(\cdot)](c_1-c_2+iu)$ is a $L^1(\mathbf{R})$ function and $|e^{((c_1-c_2)+iu)t}|\le e^{(c_1-c_2)t}$. Additionally, it holds  that \linebreak ${|e^{(c_2+iv)t}\overline{p}_H(c_2+iv,x)|\le e^{c_2t}C(c_2,x)}$. Hence
	\begin{align*}
	\left|e^{(c_2+iv)t}\overline{p}_H(c_2+iv,x)\frac{1}{2\pi}\int_{-R_1-v}^{R_1-v}e^{((c_1-c_2)+iu)t}\cL[V'_{2,H}(\cdot)](c_1-c_2+iu)du\right|\\\le e^{c_1t}\int_{-\infty}^{+\infty}|\cL[V'_{2,H}(\cdot)](c_1-c_2+iu)|du<+\infty.
	\end{align*}
	In particular we have from \eqref{PWthm}, since the integral is finite,
	\begin{equation*}
	\frac{1}{2\pi}\int_{-\infty}^{+\infty}e^{((c_1-c_2)+iu)t}\cL[V'_{2,H}(\cdot)](c_1-c_2+iu)du=V'_{2,H}(t).
	\end{equation*}
	Thus, by dominated convergence theorem, we can take the limit as \linebreak ${R_1 \to +\infty}$ in equation \eqref{eq:intpass0}, to obtain
	\begin{align}\label{eq:intpass1}
	\begin{split}
	\lim_{R_1 \to +\infty}I(R_1,R_2)=2\pi V'_{2,H}(t)\int_{-R_2}^{R_2}e^{(c_2+iv)t}\overline{p}_H(c_2+iv,x)dv.
	\end{split}
	\end{align}
	Now,  since $p_H(t,x)$ is Lipschitz, we can use the complex inversion theorem (see \cite[Theorem $2.3.4$]{wolfgang2002vector}) together with Equation \eqref{eq:Laptrpprime}  to state that
	\begin{equation*}
	\lim_{R_2 \to +\infty}\frac{1}{2\pi}\int_{-R_2}^{R_2}e^{(c_2+iv)t}\overline{p}_H(c_2+iv,x)dv=p_H(t,x).
	\end{equation*}
	Hence, taking the limit as $R_2 \to +\infty$ in equation \eqref{eq:intpass1} we have
	\begin{align*}
	\begin{split}
	\lim_{R_2 \to +\infty}\lim_{R_1 \to +\infty}I(R_1,R_2)=4\pi^2V'_{2,H}(t)p_H(t,x).
	\end{split}
	\end{align*}
	Now, taking the Laplace transform on both sides, we conclude the proof.
\end{proof}
Now let us observe that, from Proposition \ref{prop1}, we have
\begin{equation*}
p_H^\Psi(t,x)=\int_0^{+\infty}p_H(y,x)f_E(t,y)dy,
\end{equation*}
and then, taking the Laplace transform, for any $\lambda \in \mathbf{H}^*$ we have
{\small \begin{equation}\label{eq:Laptransp1}
\overline{p}_H^\Psi(\lambda,x):=\cL[p_H^\Psi(\cdot,x)](\lambda)=\int_0^{+\infty}p_H(y,x)\frac{\Psi(\lambda)}{\lambda}e^{-y\Psi(\lambda)}dy=\frac{\Psi(\lambda)}{\lambda}\overline{p}_H(\Psi(\lambda),x).
\end{equation}}
Let us observe that since $\Psi$ is a Bernstein function, then for $\lambda \in \mathbf{R}$ with $\lambda>0$ we have $\Psi'(\lambda)\ge 0$. However, $\Psi$ can be extended uniquely as an holomoprhism to $\mathbf{H}^*$, hence the zeros of $\Psi'$ are isolated (and in particular create at most a countable set). This means in particular that if $\Psi'(\lambda)=0$ for some $\lambda \in \mathbf{R}$ with $\lambda>0$, then there exists a neighborhood $U$ of $\lambda$ such that $\Psi'(\xi)>0$ for any $\xi \in U \setminus\{\lambda\}$, hence $\Psi$ is strictly increasing on the positive real axis.  In particular, since $\Psi'$ is holomorphic and completely monotone on the positive real axis, $\Psi''(\lambda)\le 0$ for any $\lambda>0$ and its zeros are isolated. Hence $\Psi'$ is strictly decreasing and then, being $\Psi'(\lambda)\ge 0$ for $\lambda>0$, it cannot achieve $0$. This guarantees invertibility of $\Psi$ over the positive real axis, but it is not enough to guarantee invertibility of $\Psi$ all over $\mathbf{H}^*$. However, since the set  of the zeros $Z$ of $\Psi'$ is at most countable, then $\Psi$ is local invertible in any $\lambda \in \mathbf{H}^*\setminus Z$. Fix then $\lambda \in \mathbf{H}^*\setminus Z$ and consider a local inverse of $\Psi$ (let us denote it by $\Psi^{-1}$). From \eqref{eq:Laptransp1} we have that
\begin{equation}\label{eq:Laptransp2}
\overline{p}_H(\lambda,x)=\frac{\Psi^{-1}(\lambda)}{\lambda}\overline{p}_H^\Psi(\Psi^{-1}(\lambda),x).
\end{equation}
In particular the quantity $\frac{\Psi^{-1}(\lambda)}{\lambda}\overline{p}_H^\Psi(\Psi^{-1}(\lambda),x)$ is independent from the choice of the local inverse map of $\Psi$.\par
Now let us define another operator.  Fix $c_1<0<c_2$ such that $c_1-c_2>-\frac{1}{\theta}$, as done in Proposition \ref{proposition:59}.  Consider $v:\mathbf{H}^* \times I \to \mathbf{R}$ where $I \subseteq \mathbf{R}$. Then we say that $v \in \mathcal{D}(\hat{L},I)$ if for all $\lambda \in \mathbf{H}^*$, $x \in I$ and  for any $z \in \mathbf{R}$ such that $c_2+iz \not \in Z$, the quantity $\frac{\Psi^{-1}(c_2+iz)}{c_2+iz}v(\Psi^{-1}(c_2+iz),x)$ is independent from the choice of the local inverse map $\Psi^{-1}$ of $\Psi$ in $c_2+iz$ and
{\small \begin{align*}
\begin{split}
\hat{L}v(\lambda&,x):=\frac{1}{4\pi^2}\int_{0}^{+\infty}e^{-\Psi(\lambda) t}\lim_{R_2 \to +\infty}\int_{-\infty}^{+\infty}e^{(c_1+iw)t}\times \\&\int_{-R_2}^{R_2}\cL[V'_{2,H}(\cdot)](c_1-c_2+i(w-z))\frac{\Psi^{-1}(c_2+iz)}{c_2+iz}v(\Psi^{-1}(c_2+iz),x)dzdwdt
\end{split}
\end{align*}}
is well defined, i.e., is  some complex number.
Recall that the  integrand  is well posed since $\frac{\Psi^{-1}(c_2+iz)}{c_2+iz}v(\Psi^{-1}(c_2+iz),x)$ is well defined except for $z \in \mathbf{R}$ such that $c_2+iz \in Z$, which is at most a countable (and then of null measure) set.
Now we are ready to prove the following lemma.
\begin{lem}
	$\overline{p}_H^\Psi \in \mathcal{D}(\hat{L},\mathbf{R}^*)$ and
	\begin{equation}\label{eq:hatLpH}
	\hat{L}\overline{p}_H^\Psi(\lambda,x)=Lp_H(\Psi(\lambda),x).
	\end{equation}
\end{lem}
\begin{proof}
	Let us just prove equation \eqref{eq:hatLpH}: the fact that $\overline{p}_H^\Psi \in   \mathcal{D}(\hat{L},\mathbf{R})$ will follow from that.  First of all, recall that from \eqref{eq:Laptransp2} we have that $\frac{\Psi^{-1}(c_2+iv)}{c_2+iv}\overline{p}_H^\Psi(\Psi^{-1}(c_2+iv),x)$ is defined independently from the choice of the local inverse $\Psi^{-1}$ of $\Psi$ on the imaginary line $\Re(\lambda)=c_2$. Moreover, local integrability of $\frac{\Psi^{-1}(c_2+iv)}{c_2+iv}\overline{p}_H^\Psi(\Psi^{-1}(c_2+iv),x)$ follows from the local integrability of $\overline{p}_H(c_2+iv,x)$.  Thus, we have  from equations \eqref{eq:Laptransp2} and \eqref{eq:LapcarL} that
	{\small \begin{align*}
	\hat{L}\overline{p}_H^\Psi&(\lambda,x)=\frac{1}{4\pi^2}\int_{0}^{+\infty}e^{-\Psi(\lambda) t}\lim_{R_2 \to +\infty}\int_{-\infty}^{+\infty}e^{(c_1+iw)t}\times \\&\int_{-R_2}^{R_2}\cL[V'_{2,H}(\cdot)](c_1-c_2+i(w-v))\overline{p}_H(c_2+iv,x)dvdwdt=Lp_H(\Psi(\lambda),x).
	\end{align*}}
\end{proof}

\section{The generalized Fokker-Plack equation for the pdf of $U_H^\Psi(t)$}\label{gFP}

In this section we want to show that $p_H^\Psi(t,x)$ is a solution (in the  sense that we will specify later) of a generalized Fokker-Planck equation.\par
Let us consider the L\'evy measure $\nu$ of the subordinator $\sigma$ and let us define $\nu_\infty(t)=\nu(t,+\infty)$ for $t>0$. The well-definition of such function is given by \eqref{intest}. Indeed, for $t>1$,
\begin{equation*}
\nu(t,+\infty)=\int_{t}^{+\infty}\nu(dx)=\int_{t}^{+\infty}(1 \wedge x)\nu(dx)\le \int_{0}^{+\infty}(1 \wedge x)\nu(dx)<+\infty
\end{equation*}
Instead, for $t \in (0,1)$, we have $\int_1^{+\infty}\nu(dx)<+\infty$ as before, while
\begin{equation*}
\int_{t}^{1}\nu(dx)=\int_{t}^{1}\frac{x}{x}\nu(dx)\le \frac{1}{t}\int_{t}^{1}x\nu(dx)=\frac{1}{t}\int_{t}^{1}(1 \wedge x)\nu(dx)<+\infty.
\end{equation*}
As done in \cite{gajda2014fokker,toaldo2015convolution}, let us define the generalized Caputo derivative of a function $v$ in a certain function space as
\begin{equation*}
\partial_t^\Psi v(t)=\der{}{t}\int_0^t\nu_\infty(t-\tau)(v(\tau)-v(0))d\tau.
\end{equation*}
Let us denote with $\mathcal{D}(\partial_t^\Psi)$ the domain of such operator. Recall that if we consider $\sigma$ to be an $\alpha$-stable subordinator for $\alpha \in (0,1)$, then the generalized Caputo derivative is actually the classical Caputo derivative of order $\alpha$. In general one can consider the fractional order of the generalized Caputo derivative depending on the asymptotic behaviour of the Levy measure $\nu$. For instance if we consider a relativistic $\alpha$-stable subordinator for $\alpha \in (0,1)$, then the memory kernel $\nu_\infty$ acts as a power of order $\alpha$ for short times and like a decreasing exponential for long times.\par
In particular, let us recall that, for $\lambda \in \mathbf{H}^*$, $\cL[\nu_\infty(\cdot)](\lambda)=\frac{\Psi(\lambda)}{\lambda}$,
and for any $v \in \mathcal{D}(\partial_t^\Psi)$ such that $\overline{v}(\lambda):=\cL[v(\cdot)](\lambda)$ and $\cL[\partial_t^\Psi v(\cdot)](\lambda)$ are defined for $\lambda \in \mathbf{H}^*$, it holds that
\begin{equation*}
\cL[\partial_t^\Psi v(\cdot)](\lambda)=\Psi(\lambda)\overline{v}(\lambda)-\frac{\Psi(\lambda)}{\lambda}v(0).
\end{equation*}
Finally, let us also recall that if $v \in C^1(0,+\infty)$, then
\begin{equation*}
\partial_t^\Psi v(t)=\int_0^t\nu_\infty(t-\tau)\der{v}{t}(\tau)d\tau.
\end{equation*}
Now, consider a function $v:[0,+\infty)\times I \mapsto \mathbf{R}$ such that the Laplace transform $\overline{v}(\lambda,x):=\cL[v(\cdot,x)](\lambda)$ is well defined for $x \in I$ and $\lambda \in \mathbf{H}^*$. Suppose moreover that $\overline{v}(\lambda,x) \in \mathcal{D}(\hat{L},I)$ and $\pdsup{}{x}{2}\hat{L}\overline{v}(\lambda,x)$ is well defined for any $x \in I$ and $\lambda \in \mathbf{H}^*$. Finally, suppose that $\pdsup{}{x}{2}\hat{L}\overline{v}(\lambda,x)$ is the Laplace transform of some $L^\infty$ function. On such function we can define the operator
\begin{equation*}
Gv(t,x)=\left[\nu_\infty \ast \cL^{-1}_{\lambda \to t}\left[\pdsup{}{x}{2}\hat{L}\overline{v}(\lambda,x)\right]\right](t);
\end{equation*}
let us denote with $\mathcal{D}(G,I)$ the domain of $G$.\par
Now we are in position to prove that $p_H^\Psi$ is a solution, in some sense, of the following equation
\begin{equation}\label{eq:genFP}
\partial_t^\Psi v(t,x)=\frac{1}{2}Gv(t,x), \quad (t,x)\in (0,+\infty)\times I
\end{equation}
for some $I \subseteq \mathbf{R}$. Following the lines of \cite{gajda2015time,meerschaert2011distributed}, we will first introduce a weaker notion of solution and then a stronger one.
\begin{defn}
	We say that $v:(0,+\infty)\times I \to \mathbf{R}$ with $I \subset \mathbf{R}$ is a \textit{mild solution} of equation \eqref{eq:genFP} if
	\begin{itemize}
		\item $v$ admits a Laplace transform $\overline{v}(\lambda,x)=\cL[v(\cdot,x)](\lambda)$ for $\lambda \in \mathbf{H}^*$;
		\item its Laplace transform $\overline{v}(\lambda,x)$ is solution of
		\begin{equation}\label{eq:genLFP}
		\Psi(\lambda)\overline{v}(\lambda,x)-\frac{\Psi(\lambda)}{\lambda}v(0,x)=\frac{\Psi(\lambda)}{2\lambda}\pdsup{}{x}{2}\hat{L}\overline{v}(\lambda,x), \ \lambda \in \mathbf{H}^*, \ x \in I.
		\end{equation}
	\end{itemize}
\end{defn}
\begin{defn}
	We say that $v:(0,+\infty)\times I \to \mathbf{R}$ with $I \subseteq \mathbf{R}$ is a \textit{classical solution} of equation \eqref{eq:genFP} if
	\begin{itemize}
		\item $v \in \mathcal{D}(G,I)$;
		\item $v(\cdot,x) \in \mathcal{D}(\partial_t^\Psi)$ for any $x \in I$;
		\item for any $x \in I$ the identity \eqref{eq:genFP} holds for almost any $t \in (0,+\infty)$.
	\end{itemize}
\end{defn}
Practically a mild solution is a \textit{weaker} solution of equation \eqref{eq:genFP}. Indeed, we say that $v$ is a mild solution if it is Laplace transformable and its Laplace transform solves equation \eqref{eq:genLFP}, that is the family of second order ordinary differential equations (in $x$) obtained from \eqref{eq:genFP} by formally applying the Laplace transform in $t$ on both sides of the equation. In particular if $v$ is a classical solution of \eqref{eq:genFP} in $(0,+\infty)\times I$, then it is Laplace transformable (since $v \in \mathcal{D}(G,I)$). Thus, if $\partial^\Phi_t v$ is Laplace transformable, we can really apply the Laplace transform in $t$ on both sides of \eqref{eq:genFP} obtaining Equation \eqref{eq:genLFP}. This means that if $v$ is a classical solution of \eqref{eq:genFP} and $\partial^\Phi_t v$ is Laplace transformable, then $v$ is also a mild solution of \eqref{eq:genLFP}.\par 
Now let us show in which sense $p_H^\Psi(t,x)$ is solution of \eqref{eq:genFP}.
\begin{thm}
	$p_H^\Psi(t,x)$ is a mild solution of equation \eqref{eq:genFP} in \linebreak ${(0,+\infty)\times \mathbf{R}^*}$.
\end{thm}
\begin{proof}
	First of all, let us observe that we have $\overline{p}_H^\Psi \in \mathcal{D}(\hat{L},\mathbf{R}^*)$. Moreover, for $x \in \mathbf{R}^*$ we have, by using equation \eqref{eq:hatLpH} and Lemma \ref{lemma:deriv} that
	\begin{equation*}
	\pdsup{}{x}{2}\hat{L}\overline{p}_H^\Psi(\lambda,x)=\pdsup{}{x}{2}Lp_H(\Psi(\lambda),x)=L\left(\pdsup{p_H}{x}{2}\right)(\Psi(\lambda),x).
	\end{equation*}
	Now let us recall that $p_H(t,x)$ is a solution of the following equation:
	\begin{equation*}
	\pd{p_H}{t}(t,x)=\frac{1}{2}V'_{2,H}(t)\pdsup{p_H}{x}{2}(t,x).
	\end{equation*}
	Now, since we have shown in Lemma \ref{lemma:deriv} that the right hand side of this equation is transformable for $x \not = 0$, let us consider its Laplace transform   for $\lambda \in \mathbf{H}^*$ and $x \in \mathbf{R}^*$:
	\begin{equation*}
	\lambda \overline{p}_H(\lambda,x)-p_H(0,x)=\frac{1}{2}\pdsup{}{x}{2}Lp_H(\lambda,x).
	\end{equation*}
	Substituting $\Psi(\lambda)$ in place of $\lambda$ we have  from equation \eqref{eq:hatLpH} that 
	\begin{equation*}
	\Psi(\lambda)\overline{p}_H(\Psi(\lambda),x)-p_H(0,x)=\frac{1}{2}\pdsup{}{x}{2}Lp_H(\Psi(\lambda),x)=\frac{1}{2}\pdsup{}{x}{2}\hat{L}\overline{p}^\Psi_H(\lambda,x).
	\end{equation*}
	Now let us use equation \eqref{eq:Laptransp1} to obtain
	\begin{equation*}
	\lambda\overline{p}_H^\Psi(\lambda,x)-p_H(0,x)=\frac{1}{2}\pdsup{}{x}{2}\hat{L}\overline{p}^\Psi_H(\lambda,x).
	\end{equation*}
	Now, since $E(0)=0$ almost surely, we have $p_H(0,x)=p_H^\Psi(0,x)$ and then
	\begin{equation*}
	\lambda\overline{p}_H^\Psi(\lambda,x)-p_H^\Psi(0,x)=\frac{1}{2}\pdsup{}{x}{2}\hat{L}\overline{p}^\Psi_H(\lambda,x).
	\end{equation*}
	Finally, multiplying by $\frac{\Psi(\lambda)}{\lambda}$, we have
	\begin{equation*}
	\Psi(\lambda)\overline{p}_H^\Psi(\lambda,x)-\frac{\Psi(\lambda)}{\lambda}p_H^\Psi(0,x)=\frac{\Psi(\lambda)}{2\lambda}\pdsup{}{x}{2}\hat{L}\overline{p}^\Psi_H(\lambda,x).
	\end{equation*}
\end{proof}
Now we want to investigate under which hypotheses $p^\Psi_H(t,x)$ is also a classical solution of \eqref{eq:genFP} in $(0,+\infty)\times \mathbf{R}^*$.
To do this, let us introduce the integrated tail of the L\'evy measure, defined as
\begin{equation*}
I(x)=\int_0^x \nu_\infty(y)dy.
\end{equation*}
{In particular we have, by using Fubini's theorem,
	\begin{align*}
	\int_0^x \nu_\infty(y)dy&=\int_0^x \int_y^{+\infty}\nu(dz)dy\\&=\int_0^x \int_y^{x}\nu(dz)dy+\int_0^x \int_x^{+\infty}\nu(dz)dy\\&=\int_0^x z \nu(dz)+x\nu_\infty(x)<+\infty
	\end{align*}
	hence $I(x)$ is well defined.}\par
Let us show some asymptotic properties of such function.
\begin{lem}\label{lemma:asymI}
	The following properties hold:
	\begin{itemize}
		\item[$i$] For any $\lambda \in \mathbf{H}^*$ we have $\lim_{x \to +\infty}e^{-\lambda x}I(x)=0$;
		\item[$i'$] We have $\lim_{x \to +\infty}\frac{I(x)}{x}=0$;
		\item[$ii$] We have $\lim_{x \to 0}I(x)=0$.
	\end{itemize}
\end{lem}
\begin{proof}
	Let us recall that since $\int_0^{+\infty}(1 \wedge x)\nu(dx)<+\infty$ we have that $\nu_\infty(x)<+\infty$ for any $x>1$. Without loss of generality, we can suppose that $\lambda \in \mathbf{R}$ with $\lambda>0$. Hence we have
	\begin{equation*}
	\lim_{x \to +\infty}\frac{I(x)}{e^{\lambda x}}=\lim_{x \to +\infty}\frac{\nu_\infty(x)}{\lambda e^{\lambda x}}=0.
	\end{equation*}
	Moreover, since $\lim_{x \to +\infty}\nu_\infty(x)=0$, we have
	\begin{equation*}
	\lim_{x \to +\infty}\frac{I(x)}{x}=\lim_{x \to +\infty}\nu_\infty(x)=0.
	\end{equation*}
	Now recall from \cite[Section $3.1$, Proposition $1$]{bertoin1996Levy} that there exists a constant $C$ such that
	\begin{equation*}
	I\left(\frac{1}{x}\right)\le C\frac{\Psi(x)}{x}.
	\end{equation*}
	Moreover, from \cite[Remark $3.3$]{schilling2012bernstein} and the fact that $\sigma$ is a driftless subordinator, we know that $\lim_{x \to +\infty}\frac{\Psi(x)}{x}=0$. Hence we have
	\begin{equation*}
	\limsup_{x \to +\infty}I\left(\frac{1}{x}\right)\le C \lim_{x \to +\infty}\frac{\Psi(x)}{x}=0
	\end{equation*}
	concluding the proof.
\end{proof}
To give some condition under which $p_H^\Psi(t,x)$ is a classical solution of \eqref{eq:genFP}, we need to introduce a particular function space that is the Widder space (\cite{wolfgang2002vector})
\begin{equation*}
C_W(0,+\infty)=\left\{f \in C^\infty(0,+\infty): \ \Norm{f}{W}<+\infty\right\}
\end{equation*}
where
\begin{equation*}
\Norm{f}{W}=\sup_{\lambda>0, \ k \in \N}\frac{\lambda^{k+1}}{k!}|f^{(k)}(\lambda)|<+\infty.
\end{equation*}
For a function $f:\mathbf{H}^* \to \mathbf{R}$, we say that $f$ belongs to the Widder space $C_W(0,+\infty)$ if and only if its restriction to the real line $\mathbf{R}_+$ belongs to $C_W(0,+\infty)$.\par
A classical result (see \cite{wolfgang2002vector} also for generalization) shows that the Laplace transform is an isometric isomorphism between $L^\infty(0,+\infty)$ and $C_W(0,+\infty)$.\par
Now we are ready to prove the following result.
\begin{prop}
	Suppose that for any $x \in \mathbf{R}^*$ the function \linebreak $\pdsup{}{x}{2}\hat{L}\overline{p}_H^\Psi(\cdot,x)$ belongs to the Widder space $C_W(0,+\infty)$. Then $p_H^\Psi(t,x)$ is a classical solution of Equation \eqref{eq:genFP} for $(t,x)\in (0,+\infty)\times \mathbf{R}^*$.
\end{prop}
\begin{proof}
	First of all, let us show that $p_H^\Psi$ belongs to $\mathcal{D}(G,\mathbf{R}^*)$. To do this, let us just observe that since $\pdsup{}{x}{2}\hat{L}\overline{p}_H^\Psi(\cdot,x)$ belongs to the Widder space $C_W(0,+\infty)$, then $\pdsup{}{x}{2}\hat{L}\overline{p}_H^\Psi(\cdot,x)$ is the Laplace transform of some $L^\infty(0,+\infty)$ function. Moreover, since $\nu_\infty$ is in $L^1_{loc}(0,+\infty)$, the convolution product in $G$ is well defined. Hence $p_H^\Psi$ belongs to $\mathcal{D}(G,\mathbf{R}^*)$. Moreover, for any $x \in \mathbf{R}^*$ there exists a constant $C_H(x)$ such that
	\begin{equation*}
	\left|\cL^{-1}_{\lambda \to t}\left[\pdsup{}{x}{2}\hat{L}\overline{p}_H^\Psi(\cdot,x)\right](t)\right|\le C_H(x).
	\end{equation*}
	Now, let us recall, by property $i'$ of the previous lemma, that there exists a constant $C_1$ such that $I(t)\le C_1t$ for any $t \ge 1$. Hence we have, for $t \ge 1$
	\begin{equation*}
	|Gp_H^\Psi(t,x)|\le I(t)C_H(x)\le C_1C_H(x)t
	\end{equation*}
	and in particular there exists a constant $C_2>0$ such that
	\begin{equation*}
	\left|\int_0^tGp_H^\Psi(s,x)ds\right|\le C_2C_H(x)t^2.
	\end{equation*}
	We have then that the function $t \mapsto \int_0^t Gp_H^\Psi(s,x)ds$ is Laplace transformable and the Laplace transform is given by (since $Gp_H^\Psi$ is by definition Laplace transformable)
	\begin{equation*}
	\cL\left[\int_0^\cdot Gp_H^\Psi(s,x)ds\right](\lambda)=\frac{1}{\lambda}\cL[Gp_H^\Psi(\cdot,x)](\lambda)=\frac{\Psi(\lambda)}{\lambda^2}\pdsup{}{x}{2}\hat{L}\overline{p}_H^\Psi(\lambda,x).
	\end{equation*}
	Now, since $p_H^\Psi(t,x)$ is mild solution of \eqref{eq:genFP}, we have that (by multiplying \eqref{eq:genLFP} by $\frac{1}{\lambda}$ and recalling that $p_H^\Psi(0,x)=0$ for any $x \in \mathbf{R}^*$)
	\begin{equation*}
	\frac{\Psi(\lambda)}{\lambda}\overline{p}_H^\Psi(\lambda,x)=\frac{\Psi(\lambda)}{2\lambda^2}\pdsup{}{x}{2}\hat{L}\overline{p}_H^\Psi(\lambda,x)
	\end{equation*}
	thus, by bijectivity of the Laplace transform, we have
	\begin{equation*}
	(\nu_\infty \ast p_H^\Psi(\cdot,x))(t)=\frac{1}{2}\int_0^t Gp_H^\Psi(s,x)ds.
	\end{equation*}
	Finally, since we have that $|Gp_H^\Psi(t,x)|\le C_1 C_H(x)t$, it is a $L^1_{loc}(0,+\infty)$ function, hence $(\nu_\infty \ast p_H^\Psi(\cdot,x))(t)$ is an absolutely continuous function and we can take the derivative almost everywhere, concluding the proof.
\end{proof}
   \vspace*{-12pt}
\subsection{A revisited generalized Fokker-Planck equation}\label{revisited}
However, asking for $\pdsup{}{x}{2}\hat{L}\overline{p}_H^\Psi(\cdot,x)$ to be in $C_W(0,+\infty)$ could be quite a strong assumption. To overcome such problem, let us introduce another operator. Consider a function $v:(0,+\infty)\times I \mapsto \mathbf{R}$ for some $I \subseteq \mathbf{R}$ such that the Laplace transform $\overline{v}(\lambda,x)=\cL[v(\cdot,x)](\lambda)$ is well defined for $\lambda \in \mathbf{H}^*$. Suppose moreover that $\overline{v}(\lambda,x) \in \mathcal{D}(\hat{L},I)$ and $\pdsup{}{x}{2}\hat{L}\overline{v}(\lambda,x)$ is well defined for any $\lambda \in \mathbf{H}^*$ and $x \in I$. Finally, suppose that the function $\lambda \in \mathbf{H}^* \mapsto \frac{\Psi(\lambda)}{\lambda}\pdsup{}{x}{2}\hat{L}\overline{v}(\lambda,x)$ is the Laplace transform of some function. Then we can define on such function the operator
\begin{equation*}
\cG v(t,x)=\cL^{-1}_{\lambda \to t}\left[\frac{\Psi(\lambda)}{\lambda}\pdsup{}{x}{2}\hat{L}\overline{v}(\lambda,x)\right](t);
\end{equation*}
let us denote with $\mathcal{D}(\cG,I)$ the domain of such operator.\par
Let us observe that by definition $\mathcal{D}(G,I)\subseteq \mathcal{D}(\cG,I)$ and if $v \in \mathcal{D}(G,I)$ then
\begin{equation*}
\cG v(t,x)=Gv(t,x).
\end{equation*}
We can now investigate the following equation
\begin{equation}\label{eq:genFP2}
\partial_t^\Psi v(t,x)=\frac{1}{2}\cG v(t,x) \ (t,x)\in (0,+\infty)\times I.
\end{equation}
Concerning the definition of solution, mild solutions of \eqref{eq:genFP} and \eqref{eq:genFP2} coincide. However, we need to give a different definition of classical solution.
\begin{defn}
	We say that $v:(0,+\infty)\times I \to \mathbf{R}$ with $I \subseteq \mathbf{R}$ is a \textit{classical solution} of Equation \eqref{eq:genFP2} if
	\begin{itemize}
		\item $v \in \mathcal{D}(\cG,I)$;
		\item $v(\cdot,x) \in \mathcal{D}(\partial_t^\Psi)$ for any $x \in I$;
		\item for any $x \in \mathbf{R}^*$ the identity \eqref{eq:genFP2} holds for almost any $t \in (0,+\infty)$.
	\end{itemize}
\end{defn}
Obviously, classical solutions of \eqref{eq:genFP} are also classical solutions of \eqref{eq:genFP2}, but the vice-versa is not true.\par
Now, we want to show that $p_H^\Psi(t,x)$ is classical solution of \eqref{eq:genFP2}. To do this, we will need the following lemma.
\begin{lem}
	Suppose that $p^\Psi_H(t,x)$ belongs to $\mathcal{D}(\partial_t^\Psi)$. Then $\partial_t^\Psi p^\Psi_H(t,x)$ is Laplace transformable.
\end{lem}
\begin{proof}
	Let us first consider $0<\varepsilon<M$, $\lambda \in \mathbf{H}^*$ and $x \in \mathbf{R}^*$. Without loss of generality, we can consider $\lambda \in \mathbf{R}$ with $\lambda>0$. We have, since $p_H^\Psi(0,x)=0$,
	\begin{multline}\label{eq:pass1class}
	\int_\varepsilon^Me^{-\lambda t}\der{}{t}\left(\int_0^t\nu_\infty(\tau)p_H^\Psi(t-\tau,x)d\tau\right)dt=\lambda(I_1(\varepsilon) -I_2(M))\\+\lambda\int_\varepsilon^{M}e^{-\lambda t}\left(\int_0^t\nu_\infty(\tau)p_H^\Psi(t-\tau,x)d\tau\right)dt,
	\end{multline}
	where
	\begin{align*}
	I_1(\varepsilon)&:=e^{-\lambda \varepsilon} \int_0^\varepsilon\nu_\infty(\tau)p_H^\Psi(\varepsilon-\tau,x)d\tau, \\ I_2(M)&:=e^{-\lambda M}\int_0^M \nu_\infty(\tau)p_H^\Psi(M-\tau,x)d\tau.
	\end{align*}
	For $I_1(\varepsilon)$, we have from Lemma \ref{lemma:contr}
	\begin{equation*}
	0\le I_1(\varepsilon)\le C_H(x)I(\varepsilon)
	\end{equation*}
	and then, taking the limit as $\varepsilon \to 0$ and using property $ii$ of Lemma \ref{lemma:asymI} we have $\lim_{\varepsilon \to 0}I_1(\varepsilon)=0$.\par
	For $I_2(M)$, we have from Lemma \ref{lemma:contr}
	\begin{equation*}
	I_2(M)\le C_H(x)e^{-\lambda M}I(M)
	\end{equation*}
	and then, taking the limit as $M \to +\infty$ and using property $i$ of Lemma \ref{lemma:asymI} we have $\lim_{M \to +\infty}I_2(M)=0$.\par
	Now, taking the limit for $\varepsilon \to 0$ and $M \to +\infty$ in Equation \eqref{eq:pass1class} we have
	\begin{align*}
	\int_0^{+\infty} e^{-\lambda t}\der{}{t}&\left(\int_0^t\nu_\infty(\tau)p_H^\Psi(t-\tau,x)d\tau\right)dt\\&=\lambda \int_0^{+\infty}e^{-\lambda t}\left(\int_0^t\nu_\infty(\tau)p_H^\Psi(t-\tau,x)d\tau\right)dt.
	\end{align*}
	Now let us consider the function
	\begin{equation*}
	I_3(\tau)=\lambda\int_{\tau}^{+\infty}e^{-\lambda t}\nu_\infty(\tau)p_H^\Psi(t-\tau,x)dt.
	\end{equation*}
	By Fubini's theorem, it is enough to show that $I_3(\tau)\in L^1((0,+\infty),d\tau)$ to conclude the proof. To do that, let us observe, by Lemma \ref{lemma:contr}, that
	\begin{equation*}
	I_3(\tau)\le C_H(x)\nu_\infty(\tau)\int_{\tau}^{+\infty}\lambda e^{-\lambda t}dt=C_H(x)\nu_\infty(\tau)e^{-\lambda \tau}
	\end{equation*}
	and
	\begin{equation*}
	\int_{0}^{+\infty}I_3(\tau)d\tau\le C_H(x)\int_0^{+\infty}\nu_{\infty}(\tau)e^{-\lambda \tau}d\tau=\frac{\Psi(\lambda)}{\lambda}
	\end{equation*}
	concluding the proof.
\end{proof}
Now we can prove the following result.
\begin{thm}
	If $p_H^\Psi(t,x) \in \mathcal{D}(\partial_t^\Psi)$ then $p_H^\Psi(t,x)$ is classical solution of \eqref{eq:genFP2} in $(0,+\infty)\times \mathbf{R}^*$.
\end{thm}
\begin{proof}
	Let us recall that $p_H^\Psi$ is mild solution of \eqref{eq:genFP}, hence
	\begin{equation*}
	\Psi(\lambda)\overline{p}_H^\Psi(\lambda,x)-\frac{\Psi(\lambda)}{\lambda}p_H^\Psi(0,x)=\frac{\Psi(\lambda)}{2\lambda}\pdsup{}{x}{2}\hat{L}\overline{p}^\Psi_H(\lambda,x).
	\end{equation*}
	Now, since $p_H^\Psi$ and $\partial_t^\Psi p_H^\Psi$ are Laplace transformable for $x \in \mathbf{R}^*$, we have that
	\begin{equation*}
	\cL[\partial_t^\Psi p_H^\Psi(\cdot,x)](\lambda)= \Psi(\lambda)\overline{p}_H^\Psi(\lambda,x)-\frac{\Psi(\lambda)}{\lambda}p_H^\Psi(0,x)
	\end{equation*}
	hence in particular $\frac{\Psi(\lambda)}{\lambda}\pdsup{}{x}{2}\hat{L}\overline{p}^\Psi_H(\lambda,x)$ is the Laplace transform of some function (that is $\partial_t^\Psi p_H^\Psi$ itself). Thus we have that $p^\Psi_H \in \mathcal{D}(\cG, \mathbf{R}^*)$. Finally, to obtain equation \eqref{eq:genFP2}, one just have to apply the inverse Laplace transform to both sides of \eqref{eq:genLFP}.
\end{proof}
\begin{rmk}
	Recall that (\cite{schilling2012bernstein}) a Bernstein function $\Psi$ is said to be special if the conjugate function $\Psi^*(\lambda)=\frac{\lambda}{\Psi(\lambda)}$ is still a Bernstein function. In such case, denoting with $\nu^*$ the L\'evy measure associated with $\Psi^*$ and with $\nu^*_\infty(t)=\nu^*(t,+\infty)$, we can defined the operator
	\begin{equation*}
	\mathfrak{I}^{\Psi^*}u(t)=(u \ast \nu^*_\infty)(t).
	\end{equation*}
	For $u$ sufficiently smooth, as stated in \cite{meerschaert2018relaxation}, the operator $\mathfrak{I}^{\Psi^*}$ acts as the inverse of $\partial^\Psi_t$, in particular
	\begin{equation*}
	\mathfrak{I}^{\Psi^*}\partial^\Psi_t u(t)=u(t)-u(0).
	\end{equation*}
	For this reason, for $x \in \mathbf{R}^*$, if $p^\Psi_H(t,x)$ belongs to $\mathcal{D}(\partial_t^\Psi)$, then $p^\Psi_H(t,x)$ is also solution of the integral equation
	\begin{equation*}
	p_H^\Psi(t,x)=\frac{1}{2}\mathfrak{I}^{\Psi^*}\mathcal{G}p_H^\Psi(t,x), \ (t,x)\in (0,+\infty)\times \mathbf{R}^*.
	\end{equation*}
	However, taking the Laplace transform of the right-hand side we have
	\begin{align*}
	\cL\left[\mathfrak{I}^{\Psi^*}\mathcal{G}p_H^\Psi(\cdot,x)\right](\lambda)&=\frac{1}{\lambda}\pdsup{}{x}{2}\widehat{L}\overline{p}_H^\Psi(\lambda,x).
	\end{align*}
	Defining the operator
	\begin{equation*}
	\mathfrak{G}u(t,x)=\cL_{\lambda \to t}^{-1}\left[\frac{1}{\lambda}\pdsup{}{x}{2}\widehat{L}\overline{u}(\lambda,x)\right]
	\end{equation*}
	where $\overline{u}(\lambda,x)=\cL[u(\cdot,x)](\lambda)$ (denoting its domain as $\mathcal{D}(\mathfrak{G})$) we have then, by taking the inverse Laplace transform, that $p_H^\Psi$ is solution of the following equation
	\begin{equation*}
	p_H^\Psi(t,x)=\frac{1}{2}\mathfrak{G}p_H^\Psi(t,x), \ (t,x)\in (0,+\infty)\times \mathbf{R}^*.
	\end{equation*}
	Moreover, if $\pdsup{}{x}{2}\widehat{L}\overline{p}_H^\Psi(\lambda,x)$ belongs to $C_W(0,+\infty)$, then
	\begin{equation*}
	\mathfrak{G}p_H^\Psi(t,x)=\int_0^{t}\cL^{-1}_{\lambda \to \tau}\left[\pdsup{}{x}{2}\widehat{L}\overline{p}_H^\Psi(\lambda,x)\right](\tau)d\tau
	\end{equation*}
	and then $p_H^\Psi$ is solution of the equation
	\begin{equation*}
	p_H^\Psi(t,x)=\frac{1}{2}\int_0^{t}\cL^{-1}_{\lambda \to \tau}\left[\pdsup{}{x}{2}\widehat{L}\overline{p}_H^\Psi(\lambda,x)\right](\tau)d\tau, \ (t,x)\in (0,+\infty)\times \mathbf{R}^*.
	\end{equation*}
	In particular, $p_H^\Psi$ is an absolutely continuous function and it solves, for any $x \in \mathbf{R}^*$ and almost every $t>0$, the equation
	\begin{equation*}
	\der{}{t}p_H^\Psi(t,x)=\frac{1}{2}\cL^{-1}_{\lambda \to t}\left[\pdsup{}{x}{2}\widehat{L}\overline{p}_H^\Psi(\lambda,x)\right](t).
	\end{equation*}
\end{rmk}

\section*{Acknowledgements}
 We are thankful to the referees for their fruitful suggestions.    The second author was partially supported by the project  STORM: Stochastics for Time-Space Risk Models, funded by the University of Oslo and the Research Council of Norway within the ToppForsk call, number 274410. The first and third author are partially supported by MIUR - PRIN 2017, project "Stochastic Models for Complex Systems", no. 2017JFFHSH.



\begin{thebibliography}{99}
 \normalsize
\bibitem{anh2005financial}
V. Anh, A. Inoue,
Financial markets with memory I: Dynamic models.
\emph{Stoch. Anal. Appl.} \textbf{23}, No 2 (2005), 275--300; doi:10.1081/SAP-200050096 
\bibitem{mcap2019}
G. Ascione, Y. Mishura, E. Pirozzi,
A fractional Ornstein-Uhlenbeck process with a stochastic forcing term and its applications.
\emph{Methodol. Comput. Appl.}, (2019), 1--32; doi: 10.1007/s11009-019-09748-y %
\bibitem{mbe2019}
G. Ascione, E. Pirozzi,
On a stochastic neuronal model integrating correlated inputs.
\emph{Math. Biosci. Eng.} \textbf{16}, No 5 (2019), 5206--5225; doi: 10.3934/mbe.2019260 %
\bibitem{ascione2017exit}
G. Ascione, E. Pirozzi, B. Toaldo,
On the exit time from open sets of some semi-Markov processes.
\emph{To appear in Ann. Appl. Probab.}, (2019); arXiv:1709.06333 %
\bibitem{ascione2019} 
G. Ascione, B. Toaldo,
A Semi-Markov Leaky Integrate-and-Fire Model. 
\emph{Math.} \textbf{7}, (2019); doi:10.3390/math7111022 
\bibitem{bertoin1996Levy}
J. Bertoin,
\emph{L\'evy Processes}.
Cambridge University Press, Cambridge (1996).
\bibitem{bertoin1999subordinators}
J. Bertoin,
\emph{Subordinators: Examples and Applications}.
Springer, Berlin, (1999).
\bibitem{brouste2013parameter}
A. Brouste, S. M. Iacus,
Parameter estimation from the discretely observed fractional Ornstein-Uhlenbeck process and the Yuima R package.
\emph{Computation. Stat.} \textbf{28}, No 4 (2013), 1529--1547; doi:10.1007/s00180-012-0365-6
\bibitem{butko2018} 
Y. A. Butko,
Chernoff approximation for semigroups generated by killed Feller processes and Feynman formulae for time-fractional Fokker–Planck–Kolmogorov equations.
\emph{Fract. Calc. Appl. Anal.} \textbf{21}, No 5 (2018), 1203--1237; doi:10.1515/fca-2018-0065
\bibitem{cahoy2014parameter}
D. O. Cahoy, F. Polito,
Parameter estimation for fractional birth and fractional death processes.
\emph{Stat. Comput.} \textbf{24}, No 2 (2014), 211--222; doi:10.1007/s11222-012-9365-1
\bibitem{cahoy2015transient}
D. O. Cahoy, F. Polito, V. Phoha,
Transient behavior of fractional queues and related processes.
\emph{Methodol. Comput. Appl.} \textbf{17}, No 4 (2015), 739--759; doi:10.1007/s11009-013-9391-2
\bibitem{cheridito2003fractional}
P. Cheridito, H. Kawaguchi, M. Maejima,
Fractional Ornstein-Uhlenbeck processes
\emph{Electron. J. Probab.} \textbf{8}, (2003), 3--14; doi:10.1214/EJP.v8-125
\bibitem{cinlar1974markov}
E. Cinlar,
Markov additive processes and semi-regeneration.
\emph{Technical report, North-western University, Center of Mathematical Studies in Economics and Management Science}, (1974).
\bibitem{delorme2015maximum}
M. Delorme, K. J. Wiese,
Maximum of a fractional Brownian motion: Analytic results from perturbation theory.
\emph{Phys. Rev. Lett.} \textbf{115}, No 21 (2015); doi:10.1103/PhysRevLett.115.210601
\bibitem{dovidio2018} 
M. D’Ovidio, S. Vitali, V. Sposini, O. Sliusarenko, P. Paradisi, G. Castellani, G. Pagnini,
Centre-of-mass like superposition of Ornstein–Uhlenbeck processes: A pathway to non-autonomous stochastic differential equations and to fractional diffusion. \emph{Fract. Calc. Appl. Anal.} \textbf{21}, No 5 (2018), 1420--1435; doi:10.1515/fca-2018-0074
\bibitem{gajda2014fokker}
J. Gajda, A. Wy\l{}oma\'nska,
Fokker-Planck type equations associated with fractional Borniwn motion controlled by infinitely divisible processes.
\emph{Physica A} \textbf{405}, (2014), 104--113; doi:10.1016/j.physa.2014.03.16
\bibitem{gajda2015time}
J. Gajda, A. Wy\l{}oma\'nska,
Time-changed Ornstein-Uhlenbeck process.
\emph{J. Phys. A--Math. Theor.} \textbf{48}, No13 (2015); doi:10.1088/1751-8113/48/13/135004
\bibitem{gatheral2018volatility}
J. Gatheral, T. Jaisson, M. Rosenbaum,
Volatility is rough
\emph{Quant. Financ.} \textbf{18}, No 6 (2018), 933--949; doi:10.1080/14697688.2017.1393551
\bibitem{gu2012time}
H. Gu, J. R. Liang, Y. X. Zhang,
Time-changed geometric fractional Brownian motion and option pricing with transaction costs.
\emph{Physica A} \textbf{405}, No 15 (2012), 3971--3977; doi:10.1016/j.physa.2012.03.020
\bibitem{guo2014pricing}
Z. Guo, H. Yuan,
Pricing European option under the time-changed mixed Brownian-fractional Brownian model.
\emph{Physica A} \textbf{406}, (2014), 73--79; doi:10.1016/j.physa.2014.03.032
\bibitem{hah2011}
M. G. Hahn, K. Kobayashi, J. Ryvkina, S. Umarov,
On time-changed Gaussian processes and their associated Fokker-Planck Kolmogorov equations.
\emph{Electron. Commun. Prob.} \textbf{16}, (2011), 150--164; doi:10.1214/ECP.v16-1620
\bibitem{hu2010parameter}
Y. Hu, D. Nualart,
Parameter estimation from fractional Ornstein-Uhlenbeck processes.
\emph{Stat. Probabil. Lett.} \textbf{80}, No 11-12 (2010), 1030--1038; doi:10.1016/j.spl.2010.02.018
\bibitem{janssen2007semi}
J. Janssen, R. Manca,
\emph{Semi-Markov risk models for finance, insurance and reliability}.
Springer Science \& Business Media, (2007).
\bibitem{jeon2014first}
J. H. Jeon, A. V. Chechkin, and R. Metzler,
First passage behavior of multi-dimensional fractional Brownian motion and application to reaction phenomena
In: \emph{First-Passage Phenomena and Their Applications},
World Scientific, (2014), 175--202.
\bibitem{kukush2017hypothesis}
A. Kukush, Y. Mishura, K. Ralchenko,
Hypothesis testing of the drift parameter sign for fractional Ornstein-Uhlenbeck process.
\emph{Electron. J. Stat.} \textbf{11}, No 1 (2017), 385--400; doi:10.1214/17-EJS1237
\bibitem{kumar2019}
A. Kumar, J. Gajda,  A. Wy\l{}oma\'nska, R. Poloczanski,
Fractional Brownian motion delayed by tempered and inverse tempered stable subordinators.
\emph{Methodol. Comput. Appl.} \textbf{21}, No 1 (2019), 185--202; doi:10.1007/s11009-018-9648-x
\bibitem{lefevre2016sir} C. Lef\`evre, M. Simon,
SIR epidemics with stage of infection.
\emph{Adv. Appl. Probab.} \textbf{48}, No 3 (2016), 768--791; doi:10.1017/apr.2016.27
\bibitem{leonenko2013fractional} N. N. Leonenko, M. M. Meerschaert, A. Sikorskii,
Fractional Pearson diffusions.
\emph{J. Math. Anal. Appl. } \textbf{403}, No 2 (2013), 532--546; doi:10.1016/j.jmaa.2013.02.046
\bibitem{leonenko2013correlation} N. N. Leonenko, M. M. Meerschaert, A. Sikorskii,
Correlation structure of fractional Pearson diffusions.
\emph{Comput. Math. Appl.} \textbf{66}, No 5 (2013), 737--745; doi:10.1016/j.camwa.2013.01.009
\bibitem{li2008simulation}
X. Li, K. Zheng, Y. Yang,
A simulation platform of DDoS attach based on network processor.
In: \emph{2008 International Conference on Computational Intelligence and Security},
IEEE, (2008), 421--426.
\bibitem{meerschaert2011distributed} M. M. Meerschaert, E. Nane, P. Vellaisamy,
Distributed-order fractional diffusions on bounded domains.
\emph{J. Math. Anal. Appl. } \textbf{379}, No 1 (2011), 216--228; doi:10.1016/j.jmaa.2010.12.056
\bibitem{meerschaert2008triangular} M. M. Meerschaert, H. P. Scheffler,
Triangular array limits for continuous time random walks.
\emph{Stoc. Proc. Appl.} \textbf{118}, No 9 (2008), 1606--1633; doi:10.1016/j.spa.2007.10.005
\bibitem{meerschaert2011stochastic}
 M. M. Meerschaert, A. Sikorskii,
\emph{Stochastic Models for Fractional Calculus}.
Walter de Gruyter, Berlin, (2011).
\bibitem{meerschaert2013inverse} M. M. Meerschaert, P. Straka,
Inverse stable subordinators.
\emph{Math. Model. Nat. Pheno.} \textbf{8}, No 2 (2013), 1--16; doi:10.1051/mmnp/20138201
\bibitem{meerschaert2018relaxation} M. M. Meerschaert, B. Toaldo,
Relaxation patterns and semi-Markov dynamics.
\emph{Stoc. Proc. Appl.} \textbf{129}, (2018); doi:10.1016/j.spa.2018.08.004
\bibitem{mij2014} J. B. Mijena,
Correlation structure of time-changed fractional Brownian motion.
\emph{ArXiv preprint}, (2014); arxiv:1408.4502.
\bibitem{Mis2008}
Y. Mishura,
\emph{Stochastic Calculus for Fractional Brownian Motion and Related Processes}.
Springer, Berlin, (2008).
\bibitem{mishura2017stochastic} Y. Mishura, V. I. Piterbarg, K. Ralchenko, A. Yurchenko-Tytarenko,
Stochastic representation and pathwise properties of fractional Cox-Ingersoll-Ross process.
\emph{Theory Probab. Math. Stat.} \textbf{97}, (2018), 157--170; doi:10.1090/tpms/1055
\bibitem{mishura2018fractional} Y. Mishura, A. Yurchenko-Tytarenko,
Fractional Cox-Ingersoll-Ross process with non-zero mean.
\emph{Mod. Stoch. Theory Appl.} \textbf{5}, (2018), 99--111; doi:10.15559/18-VMSTA97
\bibitem{rudin2006real}
W. Rudin,
\emph{Real and Complex Analysis}.
McGraw-Hill, (2006).
\bibitem{sakai1999temporally} Y. Sakai, S. Funahashi, S. Shinomoto,
Temporally correlated inputs to Leaky Integrate-and-Fire models can reproduce spiking statistics of cortical neurons.
\emph{Neural Networks} \textbf{12}, No 7--8 (1999), 1181--1190; doi:10.1016/s0893-6080(99)00053-2
\bibitem{schilling2012bernstein}
R. L. Schilling, R. Song, Z. Vondracek,
\emph{Bernstein Functions: Theory and Applications}.
Walter de Gruyter, Berlin, (2012).
\bibitem{shinomoto1999ornstein} S. Shinomoto, Y. Sakai, S. Funahashi, 
The Ornstein-Uhlenbeck process does not reproduce spiking statistics of neurons in prefrontal cortex.
\emph{Neural Comput.} \textbf{11}, No 4 (1999), 935--951; doi:10.1162/089976699300016511
\bibitem{tenreiro2017} J. A. Tenreiro Machado, V. Kiryakova,
The chronicles of fractional calculus. \emph{ Fract. Calc. Appl. Anal.} \textbf{20}, No 2 (2017), 307--336; doi:10.1515/fca-2017-0017
\bibitem{toaldo2015convolution} B. Toaldo,
Convolution-type derivatives, hitting-times of subordinators and time-changed $C^0$-semigroups.
\emph{Potential Anal.} \textbf{42}, No 1 (2015), 115--140; doi:10.1007/s11118-014-9426-5
\bibitem{weron2009anomalous} A. Weron, M. Magdziarz, 
Anomalous diffusion and semimartingales
\emph{Europhys. Lett.} \textbf{86}, No 6 (2009); doi:10.1209/0295-5075/86/60010
\bibitem{wolfgang2002vector}
A. Wolfgang, J. K. B. Charles, H. Matthias, N. Frank,
\emph{Vector-Valued Laplace Transform and Cauchy Problems}.
Birkhuser, Basel, (2002).

\end{thebibliography}
\end{document}